\documentclass[12pt,a4paper]{amsart}

\usepackage[english]{babel}

\usepackage[utf8]{inputenc}
\usepackage[T1]{fontenc}
\usepackage{lmodern}

\usepackage{amsthm}
\usepackage{amsmath}
\usepackage{amsfonts}
\usepackage{amssymb}
\usepackage{amscd}

\usepackage[left=2.5cm,right=2.5cm,top=3cm,bottom=2cm]{geometry}
\usepackage{hyperref}

\newtheorem{theo}{Theorem}[subsection]
\newtheorem*{theo*}{Theorem}
\newtheorem{prop}[theo]{Proposition}
\newtheorem{lemm}[theo]{Lemma}
\newtheorem{coro}[theo]{Corollary}
\newtheorem{conj}[theo]{Conjecture}

\newtheorem{defi}[theo]{Definition}

\theoremstyle{remark}
\newtheorem{rema}[theo]{Remark}

\newcommand{\Z}{\mathbf{Z}}
\newcommand{\Q}{\mathbf{Q}}
\newcommand{\F}{\mathbf{F}}

\newcommand{\Qp}{\mathbf{Q}_p}
\newcommand{\Qpb}{\overline{\mathbf{Q}}_p}
\newcommand{\cE}{\mathcal{E}}

\newcommand{\rhofonctA}{\mathcal{R}_{A,\rho}}

\newcommand{\bG}{\mathbf{G}}
\newcommand{\cG}{\mathcal{G}}

\renewcommand{\O}{\mathcal{O}}

\newcommand{\m}{\mathfrak{m}}

\newcommand{\ul}[1]{{\underline{#1}}}
\newcommand{\ov}[1]{\overline{#1}}

\newcommand{\spec}{\operatorname{Spec}}
\newcommand{\spm}{\operatorname{SpecMax}}

\newcommand{\Aut}{\operatorname{Aut}}

\newcommand{\id}{\operatorname{id}}

\newcommand{\Gal}{\operatorname{Gal}}

\newcommand{\Hom}{\operatorname{Hom}}

\newcommand{\GL}{\operatorname{GL}}
\newcommand{\PGL}{\operatorname{PGL}}
\newcommand{\SL}{\operatorname{SL}}

\newcommand{\unr}{\operatorname{unr}}
\newcommand{\ind}{\operatorname{ind}}
\newcommand{\Fil}{\operatorname{Fil}}

\newcommand{\Noe}{\operatorname{Noe}}
\newcommand{\Rep}{\operatorname{Rep}}
\newcommand{\Mod}{\operatorname{Mod}}
\newcommand{\Proj}{\operatorname{Proj}}
\newcommand{\Vect}{\operatorname{Vect}}

\newcommand{\DdRA}{\operatorname{D}_{\mathrm{dR},A}}
\newcommand{\DdRArho}{\operatorname{D}_{\mathrm{dR},A,\rho}}
\newcommand{\DdR}{\operatorname{D}_{\mathrm{dR},A,\rho}}
\newcommand{\BdR}{\operatorname{B}_{\mathrm{dR}}}

\newcommand{\cyc}{\chi_{\text{cyc}}}
\newcommand{\cc}{\mathcal{C}}

\newcommand{\ssl}{\sigma_{\SL_n}}
\newcommand{\ssd}{\sigma_{\SL_2}}
\newcommand{\sgl}{\sigma_{\GL_n}}
\newcommand{\sgd}{\sigma_{\GL_2}}

\newcommand{\s}{\mathcal{S}}
\newcommand{\e}{\mathcal{E}}
\newcommand{\X}{\mathcal{X}}

\newcommand{\I}{\mathcal{I}}
\newcommand{\W}{\mathcal{W}}
\newcommand{\Gg}{\mathcal{G}}
\newcommand{\Qq}{\mathcal{Q}}
\newcommand{\Ii}{\ul{\mathcal{I}}}
\newcommand{\Ww}{\ul{\mathcal{W}}}

\title[Deformations rings for $PGL_n$]{Potentially semi-stable deformation rings for representations
with values in $\PGL_n$}

\author{Agnès David}
\address{Laboratoire de Mathématiques de Besançon\\
UMR 6623 du CNRS\\
Université de Franche-Comté\\ 
16 route de Gray, 25030 Besançon Cedex\\
France \\
\newline\indent
Université de Rennes\\
CNRS, IRMAR - UMR 6625\\ 
F-35000 Rennes\\ 
France}
\email{david.agnes@math.cnrs.fr}

\author{Sandra Rozensztajn}
\address{UMPA, \'ENS de Lyon\\
UMR 5669 du CNRS\\
46, all\'ee d'Italie\\
69364 Lyon Cedex 07\\
France}
\email{sandra.rozensztajn@ens-lyon.fr}

\begin{document}


\begin{abstract}
We study the potentially semi-stable deformation rings for Galois
representations taking their values in $\PGL_n$, by comparing them to the
deformation rings for $\GL_n$. As an application, we state an analogue of
the Breuil-Mézard conjecture, and we show that the case of $\PGL_n$
follows from the case of $\GL_n$.
\end{abstract}

\maketitle

\tableofcontents

\section{Introduction}

Let $p$ be a prime number, and let $E$ be a finite extension of $\Q_p$.
Let $G$ be a reductive group defined over the ring of integers $\O_E$ of
$E$. Rings classifying potentially semi-stable $p$-adic Galois
representations taking their values in $G(E)$ have been defined by Balaji
in his thesis (\cite{Bal}, see also Bellovin's construction in \cite{Bel16}),
generalizing a construction of Kisin (\cite{Kis08}) for $G = \GL_n$. In
the case where $G = \GL_n$, these rings have been extensively studied, in
terms of properties of their generic fibers. There are also conjectures
on multiplicities of their special fibers, as a generalized Breuil-Mézard
conjecture (see \cite{EG14}), proved in some cases. There are also some
explicit computations of these rings for small values of $n$ and some
specific Hodge types and inertial type (see for example \cite{Sha}).

When the group $G$ is not $\GL_n$ much less is known. There are general
results in \cite{BG} and \cite{BoLe}, and some explicit
computations in \cite{KoMo}.  Our goal in this article is to better
understand these deformation rings in the case where $G = \PGL_n$, with
$n$ prime to $p$, which seems in some ways to be the easiest after the
case of $\GL_n$.

The main reason why $\PGL_n$ is easier to study is that representations
with values in $\PGL_n$ can be related to representations with values in
$\GL_n$, via lifting along the natural map $\GL_n \to \PGL_n$. This
allows us to relate deformation rings for $\PGL_n$ to deformation rings
for $\GL_n$, and hence transfer our knowledge of the rings for $\GL_n$ to
$\PGL_n$.

The tool allowing us to do this comparison is the fact that the
constructions of $p$-adic Hodge theory are functorial with respect to the
reductive group $G$. We explain this functoriality, and how it 
translates into a functoriality for potentially semi-stable deformation
rings. All the main technical tools to do this are already in Balaji's
thesis, but some care must be taken in giving a good Tannakian
construction of the various objects.

Another key point is that any Galois representation with values in
$\PGL_n$ actually lifts to a Galois representation with values in $\GL_n$
if $n$ is prime to $p$, and a potentially semi-stable representation
can be lifted to a potentially semi-stable representation.

Armed with these tools, we are able to relate the potentially semi-stable
deformation rings for $\PGL_n$ to those for $\GL_n$, showing in
particular that the generic fiber of a potentially deformation ring for
$\PGL_n$ is simply a product of some explicitely determined potentially
semi-stable deformation rings for $\GL_n$ (see Theorem
\ref{compdefrings0}). We are also able to give a partial result for the
potentially semi-stable deformation ring itself (see Theorem
\ref{compdefrings}), which has for consequence the fact that we can
completely describe the cycle attached to the special fiber of this
deformation rings in terms of those for $\GL_n$ (Theorem
\ref{cyclesdefrings}).

As an application of this description of these potentially semi-stable
deformation rings, we state Breuil-Mézard type conjectures for the cycles
attached to their special fibers (Conjecture \ref{geomBMPGLnconj}, and prove
that they are consequences of conjectures for deformation rings for
$\GL_n$, as stated in \cite{EG14} (Theorem \ref{geomBMPGLn}).  Moreover,
we are able to show that results for $\PGL_n$ can be deduced from those
for $\GL_n$. In particular, we obtain unconditional results for
representations of the absolute Galois group of $\Q_p$ with values in
$\PGL_2$ (Theorem \ref{geomBMPGL2Qp}).

Given the existence of potentially semi-stable deformation rings defined
for any connected reductive group $G$, it would be interesting to state
analogues of the Breuil-Mézard conjectures for the cycles attached to the
special fiber of these rings in larger generality.
However, the Breuil-Mézard conjectures relate the cycles
attached to the special fibers, and their multiplicities, to automorphic
multiplicities that are defined in terms of the local Langlands program,
and in particular an inertial Langlands correspondence. This
correspondence is completely known for $\GL_n$, but not for other groups.
Even for $\PGL_n$ we have to resort to some ad hoc constructions (see
Paragraph \ref{deftypesn}).  In
order to state these conjectures for other groups we would need a better
understanding of this inertial correspondence in general, although this
can probably be done for some groups $G$ beyond $\GL_n$ and $\PGL_n$.

\subsection{Plan of the article}

In Section \ref{Grepr} we study the theory of $G$-representations for $p$-adic Galois
representations, and $p$-adic Hodge theory. We recall results and
constructions by Balaji, and prove that these constructions are
functorial in the group $G$.

In Section \ref{lifts} we study the problem of lifting a $p$-adic Galois
representation from $\PGL_n$ to $\GL_n$, respecting some conditions
coming from $p$-adic Hodge theory.

In Section \ref{potringspgln} we apply the previous results to the study
of potentially semi-stable deformation rings for $\PGL_n$, and we show
that we can express these rings in terms of potentially semi-stable
deformations rings for $\GL_n$.

Finally in Section \ref{BM} we state a Breuil-Mézard conjecture for
representations with values in $\PGL_n$, and we show that these
conjectures would follow from analogous conjectures for $\GL_n$, up to
some hypothesis. In order to state this conjectures we make some
constructions related to the inertial Langlands correspondence.

\subsection{Notation}

Let $p$ be a prime number and let $n \geq 2$ an integer prime to $p$.

Let $K$ be a finite extension of $\Q_p$. We will be interested in
representations of its absolute Galois group $\Gamma_K =
\Gal(\Qpb / K)$. Let $I_K$ be the inertia subgroup, and $W_K$ the Weil
group. We denote by $q$ the
cardinality of the residue field of $K$.  
We normalize the correspondence of local class field theory so that a
uniformizer of $K$ corresponds to a geometric Frobenius.
Let $\cE$ be the set of
embeddings of $K$ in $\Qpb$ ; an element of $\cE$ will generally be
denoted by $\iota$.

Denote by $\cyc$ the $p$-adic cyclotomic character, and by $\omega$ its
reduction modulo $p$. Denote by $\unr(x)$ the unramified character that
sends a geometric Frobenius to $x$, for $x \in \Qpb^\times$ or
$\ov\F_p^\times$.

We fix another finite extension $E$ of $\Qp$, which will serve as a
coefficient field for our Galois representations. As such, we will
sometimes enlarge it to a finite extension, in order to obtain a field
that is ``large enough so that everything is defined on $E$''. We denote
by $\O_E$ its ring of integers, by $\varpi_E$ a uniformizer, and by $\F$
its residue field.

We denote by $G$ a connected reductive group in section
\ref{Grepr}. In the following sections $G$ will always be either $\GL_n$
or $\PGL_n$.

Let $u : \GL_n \to \PGL_n$ be the natural homorphism of algebraic groups.
For a ring $A$, we denote by $u_A$ the natural projection from $\GL_n(A)$
to $\PGL_n(A)$.

\subsection{Acknowledgements}

We would like to thank Robin Bartlett and Christophe Breuil for their
questions and comments on this paper. 
Both authors are members of the A.N.R. project CLap-CLap ANR-18-CE40-0026.

\section{$p$-adic Hodge theory for $G$-representations}
\label{Grepr}

\subsection{Tannakian constructions}

In this section, $E$ denotes a fixed finite extension of $\Qp$, with a
given embedding into $\ov{\Q}_p$.  Let $G$ be a connected reductive group
over $E$.

We consider the category  $\Rep_E G$ of algebraic representations of
$G$ over finite dimensional $E$-vector spaces.  Objects of
$\Rep_E G$ are pairs $(V,\pi_V)$, where $V$ is a finite dimensional
vector space over $E$ and $\pi_V$ is a homomorphism of algebraic groups
from $G$ to $\GL_V$. More generally, if $R$ is an $E$-algebra we have the
category $\Rep_R G$ of algebraic representations of
$G$ over finite rank free $R$-modules, and we have a natural functor
$\Rep_E G \to \Rep_RG$.

We denote by $\Vect_E$ the category of finite dimensional $E$-vector
spaces.  Let $R$ be an $E$-algebra. We denote by $\Mod_R$ the category of
finitely generated $R$-modules and $\Proj_R$ the subcategory of those
that are projective over $R$.

For every $R$-algebra $R'$, we denote by $\phi_R^{R'}$ the extension of
scalars from $R$ to $R'$, a functor from $\Mod_R$ to $\Mod_{R'}$, sending
$\Proj_R$ to $\Proj_{R'}$.

We denote by $ \omega^G : \Rep_E G \rightarrow \Vect_E$ the forgetful
functor. If $R$ is an $E$-algebra, we denote by $\omega^G_R$ the functor
$\phi_E^R \circ \omega^G$.  Then $\omega^G_R$ is a fiber functor on
$\Rep_EG$ with values in $\Mod_R$, that is, a $E$-linear exact faithful
tensor functor from $\Rep_E G$  to $\Mod_R$ taking values in $\Proj_R$,
and it factors through $\Rep_R G$.

Let $R$ be an $E$-algebra and $N$ a tensor functor from $\Rep_E G$ to
$\Proj_{R}$. We denote by $\ul\Aut^\otimes(N)$ the functor on
$R$-algebras sending the $R$-algebra $R'$ to the group of automorphisms
of tensor functors of $\phi_R^{R'}\circ N$.  Then
$\ul{\Aut}^\otimes(\omega^G_R) = G_R$ (\cite[Theorem 2.11]{DM})

We recall now a fundamental result from the theory of Tannakian
categories.

\begin{theo}[\cite{SaRi}, Theorem 4.2.2]
\label{funTan}
Let $\eta_1$ and $\eta_2$ be two fiber functors on $\Rep_E G$  with
values in $\Mod_R$.  There exists a faithfully flat $R$-algebra $R'$ such
that $\phi_R^{R'} \circ \eta_1$ and $\phi_R^{R'} \circ \eta_2$  are
isomorphic as tensor functors.

Let $\eta$ be a fiber functor on $\Rep_E G$ with values in $\Mod_R$.  The
functor $\underline{\Aut}^\otimes(\eta)$ is representable by an affine
group scheme over $R$, which we denote by $G^\eta$. 
An isomorphism $\alpha$ between $\phi_R^{R'}
\circ \eta$ and $\omega^G_{R'}$ induces an isomorphism of
algebraic groups between $G_{R'}$ and $G^{\eta}_{R'}$. Another
isomorphism between $\phi_R^{R'}
\circ \eta$ and $\omega^G_{R'}$ induces an isomorphism of
algebraic groups between $G_{R'}$ and $G^{\eta}_{R'}$ which is conjugate
to the first one by an element of $G_{R'}$.
\end{theo}

Following formally the definitions we get:

\begin{prop}[Functoriality property]
\label{functoriality}
Let $H$ be a smooth connected reductive group over $E$, let $f : G
\rightarrow H$ be a map of algebraic groups, so that it induces a map of
tensor categories $f^* : \Rep_EH \rightarrow  \Rep_EG$. Let $\eta$ be a
fiber functor on $\Rep_EG$ with values in $\Mod_R$.

Then $f$ induces a map of algebraic groups $f^{\eta} : G^{\eta} \rightarrow
H^{\eta \circ f^*}$ satisfying the following condition:

Let $R'$ be a faithfully flat $R$-algebra such that $\phi_R^{R'}\circ \eta$ and
$\omega^G_{R'}$ are isomorphic, and choose any such isomorphism
$\alpha : \omega^G_{R'} \to \phi_R^{R'}\circ \eta$.
Then we have the following commutative
diagram:
\[
\begin{CD}
G_{R'} @>f_{R'}>>  H_{R'} \\
@VV{\alpha}V       @VV{f^*\alpha}V  \\
G^{\eta}_{R'} @>f^{\eta}_{R'}>> H^{\eta \circ f^*}_{R'}
\end{CD}
\]
where $f^*\alpha : \omega_{R'}^G \circ f^* = \omega_{R'}^H \to \phi_R^{R'}\circ
\eta \circ f^*$ is the isomorphism coming from $\alpha$.
\end{prop}

\subsection{$p$-adic Hodge theory} 
\label{padicHodge}

Let $A$ be a local finite $E$-algebra and $F$ its residue field, so that $F$
is a finite extension of $\Q_p$.

We recall some definitions and constructions from $p$-adic Hodge theory.

\subsubsection{de Rham, semi-stable, crystalline $A$-representations of
$\Gamma_K$}

We denote by $\Rep_A \Gamma_K$ the category of finitely generated
projective $A$-modules with a continuous $A$-linear action of $\Gamma_K$.
We call an object $M$ in $\Rep_A \Gamma_K$ de Rham (resp. semi-stable,
crystalline) if it is de Rham (resp.~semi-stable, crystalline) as a
$\Qp$-(finite dimensional)-representation of $\Gamma_K$.  We then denote
by $\Rep^{\mathrm{dR}}_A \Gamma_K$ (resp.~$\Rep^{\mathrm{st}}_A
\Gamma_K$, $\Rep^{\mathrm{cr}}_A \Gamma_K$) the full subcategory of
$\Rep_A \Gamma_K$ of de Rham (resp.~semi-stable, crystalline)
representations.

All these categories are Tannakian, as the property of being de Rham
(resp.~semi-stable, crystalline) is preserved by direct sum, sub-object,
quotient, tensor product, $\Hom$, with a fiber functor given by the
underlying $A$-module (\cite{Fon94}).

\subsubsection{Modules attached to the representations}

Recall that $\BdR$ is one of Fontaine's rings of periods (\cite{Fon82}).
The action of $\Gamma_K$ on $\BdR$ extends to a continuous $A$-linear
action on $A \otimes_{\Qp} \BdR$, satisfying: $\left(A \otimes_{\Qp}
\BdR \right)^{\Gamma_K} = A \otimes_{\Qp} K$ (\cite[Prop 2.2.7
(1)]{Bal}).  For
every $M$ in $\Rep_A \Gamma_K$, this gives a structure of  $A
\otimes_{\Qp} K$-module on $\left( M \otimes_{\Qp} \BdR
\right)^{\Gamma_K}$.  We then denote by $\Mod_{A \otimes_{\Qp} K}$ the
category of $A \otimes_{\Qp} K$-modules and
$\DdRA$ the functor from
$\Rep_A \Gamma_K$ to $\Mod_{A \otimes_{\Qp} K}$
sending $M$ to
$\left( M \otimes_{\Qp} \BdR \right)^{\Gamma_K}$.
For $M$ in $\Rep_A \Gamma_K$, $\DdRA(M)$ agrees with the usual definition
of $D_{\mathrm{dR}}(M)$, where $M$ is considered as a finite dimensional
$\Qp$-representation of $\Gamma_K$; $\DdRA(M)$ here in addition keeps
track of the $A$-module structure on $M$.

The properties of the functor $\DdRA$ on the category
$\Rep^{\mathrm{dR}}_A \Gamma_K$ are given in \cite[Prop 2.2.9]{Bal}. 
If $M$ is in $\Rep^{\mathrm{dR}}_A \Gamma_K$ then $\DdRA(M)$ is
projective over $A\otimes_{\Qp}K$. Moreover $\DdRA$ is an
$A$-linear exact faithful tensor functor on $\Rep^{\mathrm{dR}}_A
\Gamma_K$, and as it takes its values in the subcategory
$\Proj_{A\otimes_{\Qp}K}$ of finitely generated projective
$A\otimes_{\Qp}K$-modules, it is a fiber functor 
on $\Rep^{\mathrm{dR}}_A
\Gamma_K$ (and similarly for the subcategories $\Rep^{\mathrm{st}}_A
\Gamma_K$ and $\Rep^{\mathrm{cr}}_A \Gamma_K$).

\subsection{$G$-objects}

\subsubsection{General definition}

We want to define what is a representation with values in $G$ satisfying
certain properties. We refer to \cite[Appendix A]{Bel16} for other examples
of such constructions.

We first give a Tannakian definition.
Let $\cc$ be a tensor category. A $G$-object with values in $\cc$ 
is an exact faithful tensor functor from $\Rep_EG$ to $\cc$.
The categories $\cc$ that we will consider are of the following form: fix
some group $\cG$, some $E$-algebra $A$, and consider representations of $\cG$
on finite-rank projective $A$-modules satisfying some properties that
make $\cc$ into a Tannakian category with fiber functor the forgetful
functor from $\cc$ to $\Mod_A$ (see examples below).


A more naive definition would be to take maps $r : \cG \to G(A)$ such that
for any algebraic representation $f : G \to \GL_n$, the map $f\circ r 
: \cG \to \GL_n(A)$ satisfy the properties we want.

We can go from the naive definition to the Tannakian one as follows:
start with $r : \cG \to G(A)$, and define the functor $\ul{r} : \Rep_EG
\to \cc$ by the following formula: let $f : G \to \GL_n$ be an algebraic
representation, then it defines $f_A : G(A) \to \GL_n(A)$, and we set
$\ul{r}(f)$ to be $f_A \circ r$.

Going in the other direction is not as straightforward. Fix some
$E$-algebra $A$ and $\cc$ a category of representations on $\cG$ on
$A$-modules as before. We denote by $\omega_\cc$ the forgetful functor
$\cc \to \Mod_A$.

\begin{prop}
\label{Gobjtonaive}
Let $\ul{r}$ be a $G$-object with values in $\cc$.
Then we can construct an inner form $G'$ of $G$ on $A$, and a representation
$r' : \cG \to G'(A)$. They satisfy the following property: 
the functors $\Rep_EG \to \cc$ given by
$\ul{r}$ and coming from $r'$ are the same.
\end{prop}

\begin{proof}
Let $\nu = \omega_\cc \circ \ul{r}$. Then $\nu$ is a fiber functor with
values in $\Mod_A$ on the category $\Rep_EG$. Let $G' =
\ul{\Aut}^\otimes(\nu)$. We see easily that any element $g$ of $\cG$
defines an element of $G'(A)$, so that we get a representation $r' : \cG
\to G'(A)$. 

The representation
$r'$ gives rise to a functor $\Rep_AG' \to \cc$ as explained before.

We also have a canonical equivalence of categories $\Rep_AG
\to \Rep_AG'$. Composing with the natural functor $\Rep_EG \to \Rep_AG$,
and the functor $\Rep_AG' \to \cc$, 
we get a functor $\Rep_EG \to \cc$, which we can compare to $\ul{r}$. 
Unravelling the definitions gives the comparison.
\end{proof}

\begin{rema}
\label{remGobjtonaive}
Observe that
the functors $\omega^G_A$ and $\nu$ are both fiber functors on $\Rep_EG$,
so by Theorem \ref{funTan} there exists a faithfully flat $A$-algebra $B$
such that $\phi_A^B\circ \nu$ and $\omega_B^G$ are isomorphic, and so
$G_B$ and $G'_B$ are isomorphic. So $r'$ gives rise to a representation 
$\cG \to G(B)$ which is well-defined up to conjugation by an element of
$G(B)$.
Note that if $A$ is an algebraically closed field then 
we can take $A=B$, so no extension is needed,
but the isomorphism is not unique so the representation into $G(A)$ is
still only defined up to conjugation.
\end{rema}

We will only apply this inverse construction in the case where $A =
\ov\Q_p$ (in Paragraph \ref{inertialtype}).
In general we will start from the naive definition and
then use the $G$-object attached to it.

\subsubsection{Example: de Rham, semi-stable, crystalline $G$-representations of
$\Gamma_K$}
\label{sssec:dRGrepGal}

Let $A$ be an $E$-algebra as in Paragraph \ref{padicHodge}.
We fix  a continuous homomorphism $\rho$  from $\Gamma_K$ to $G(A)$.
We consider the $G$-object 
$\rhofonctA : \Rep_EG  \rightarrow \Rep_A \Gamma_K$ attached to it.

We say that $\rho$ is de Rham, or a de Rham $G$-representation of
$\Gamma_K$ (resp. semi-stable, crystalline) if the functor $\rhofonctA$
takes its values in the subcategory $\Rep^{\mathrm{dR}}_A \Gamma_K$
(resp. $\Rep^{\mathrm{st}}_A \Gamma_K$, $\Rep^{\mathrm{cr}}_A \Gamma_K$).
That is, we take for the category $\cc$ one of the categories
$\Rep^{\mathrm{dR}}_A \Gamma_K$, $\Rep^{\mathrm{st}}_A \Gamma_K$,
$\Rep^{\mathrm{cr}}_A \Gamma_K$.

Equivalently $\rho$ is de Rham (resp. semi-stable, crystalline)  if and
only if there exists a faithful algebraic
$E$-representation $V$ of  $G$ such that $\rhofonctA(V)$ is de Rham
(resp. semi-stable, crystalline) as an $A$-representation of $\Gamma_K$.

We say that $\rho$ is potentially semi-stable (resp. potentially
crystalline) if there exists a finite extension $L$ of $K$ such that the
restriction $\rho_{| \Gal(\overline{K} / L)}$ is semi-stable
(resp. crystalline) as a representation of $\Gamma_L$.

\subsubsection{Example: smooth representations of $W_K$ or $I_K$}
\label{exsmooth}

Here we take $\cG = W_K$ or $\cG = I_K$, and for $\cc$ the category of
finite-dimensional representations of $\cG$ with coefficients in
$\ov\Q_p$ with open kernel. Let $\ul{r}$ be a $G$-object with values in
$\cc$, then it defines a finite-dimensional representation $r : \cG \to
G(\ov\Q_p)$, which is well-defined up to conjugation by an element of
$G(\ov\Q_p)$. Then $r$ has an open kernel, as we can see by considering a
faithful representation of $G$.

\subsection{The Hodge type of a de Rham $G$-representation of $\Gamma_K$}

Let $A$ be an $E$-algebra as in Paragraph \ref{padicHodge}.
Let $\rho$ be a de Rham representation $\Gamma_K \to G(A)$.  In
\cite[\S 2.4]{Bal}, Balaji attaches to such a  $\rho$ a Hodge type for
$G$.  We recall briefly his construction, and give another way of looking
at it.


\subsubsection{Definition of the Hodge type}

By \cite[proposition 2.3.3]{Bal}, the composition of functors
$\DdRArho := \DdRA \circ \rhofonctA$ from $\Rep_EG$ to $\Mod_{A
\otimes_{\Qp} K}$ sending
$V$ to $\left( \rhofonctA(V) \otimes_{\Qp} \BdR \right)^{\Gamma_K} =
\left( (A \otimes_K V) \otimes_{\Qp} \BdR \right)^{\Gamma_K}$
is a fiber functor.
We denote by $G'$ the affine groupe scheme over $A \otimes_{\Qp} K$
representing the functor $\ul{\Aut}^\otimes \left( \DdRArho \right)$
(note that $G'$ depends on $\rho$).

We denote by $\BdR^+$ the valuation ring of $\BdR$ and, for every $i$ in
$\Z$, $\Fil^i \BdR^+$ the $i$-th power of its maximal ideal.
For every  $i$ in $\Z$, we define the subfunctor of  $\DdR$ :
 $$
 \begin{array}{r r c l}
 \Fil^i \DdR: & \Rep_EG &\rightarrow & \Mod_{A \otimes_{\Qp} K}\\
   & V & \mapsto & \left(\rhofonctA(V) \otimes_{\Qp} \Fil^i \BdR^+ \right)^{\Gamma_K}. \\
 \end{array}
 $$
By \cite[Prop 2.4.2]{Bal}, the subfunctors $\left( \Fil^i \DdR
\right)_{i \in \Z}$ form an exact tensor filtration of $\DdR$ (see
\cite[IV \S 2.1]{SaRi} for the precise definition).
As $G$ is smooth and we are in characteristic $0$, by \cite[IV \S 2.4
Théorème and IV Prop 2.2.5]{SaRi}, the filtration $(\Fil^i \DdR)_{i \in
\Z}$ is splittable.  This means that there exists a graduation
of $\DdRArho$, compatible to tensor products, by subfunctors $(M^i
\DdRArho)_{i \in \Z}$ such that the filtration is associated to this
graduation.  

Let us fix such a splitting $(M^i \DdRArho)_{i \in \Z}$ of
$(\Fil^i \DdR)_{i \in \Z}$. Then the data of the splitting is the same as
the data of a cocharacter $\mu$ of $G'$ in the following way: for every  $A
\otimes_{\Qp} K$-algebra $B$, $z$ in $B^\times$ and $V$ in $\Rep_E G$,
$\mu(z)_V$ is the $B$-module automorphism of $$ \DdRArho(V) \otimes_{A
\otimes_{\Qp} K} B = \bigoplus_{i \in \Z}  M^i \DdRArho(V) \otimes_{A
\otimes_{\Qp} K}  B $$ acting on the factor $M^i \DdRArho(V) \otimes_{A
\otimes_{\Qp} K}  B$ by $z^i$.  

The cocharacter $\mu$ of $G'$ depends on the choice
of the splitting $(M^i \DdRArho)_{i \in \Z}$: two different splittings of
$(\Fil^i \DdR)_{i \in \Z}$ give two cocharacters that are conjugate by
$G'(A \otimes_{\Qp} K)$ (more precisely they are conjugate by an element
of the parabolic subgroup defined by the filtration).

By reduction modulo $\m_A$, the $G'(A \otimes_{\Qp} K)$-conjugacy class
$\mathrm{Cl(\mu)}_{G'(A \otimes_{\Qp} K)}$ of $\mu$ determines an unique
$G'(F \otimes_{\Qp} K)$-conjugacy class
$\mathrm{Cl(\overline{\mu})}_{G'(F \otimes_{\Qp} K)}$ of cocharacters of
$G'_{F \otimes_{\Qp} K}$ (see \cite[Prop 2.4.5]{Bal}), where $F$ is the
residue field of $A$, and so a finite extension of $E$.

Then, $\mathrm{Cl(\overline{\mu})}_{G'(F \otimes_{\Qp} K)}$  determines
an unique $G'(\overline{E} \otimes_{\Qp} K)$-conjugacy class
$\mathrm{Cl(\overline{\mu})}_{G'(\overline{E} \otimes_{\Qp} K)}$ of
cocharacters of  $G'_{\overline{E} \otimes_{\Qp} K}$.

Because $G'_{\ov{E} \otimes_{\Qp} K}$ is an inner form of $G_{\ov{E}
\otimes_{\Qp} K}$,
$\mathrm{Cl(\overline{\mu})}_{G'(\overline{E} \otimes_{\Qp} K)}$
determines an unique $G(\overline{E} \otimes_{\Qp} K)$-conjugacy class of
cocharacters of  $G_{\overline{E} \otimes_{\Qp} K}$. This is what Balaji
calls the Hodge type of $\rho$. Remark that we can also see the Hodge
type as a family indexed by $\e = \Hom(K,\ov\Q_p)$ of
$G(\ov{E})$-conjugacy classes of cocharacters of $G_{\ov{E}}$.

\begin{defi}
Let $\rho : \Gamma_K \to G(E)$ be a potentially semi-stable
representation. Its Hodge type is the family indexed by $\e$ of
$G(\ov\Q_p)$-conjugacy classes of cocharacters $\bG_{m,\ov{E}} \to
G_{\ov{E}}$ given by the construction above. 
\end{defi}

\begin{rema}
It is clear from the construction that the Hodge type of $\rho : \Gamma_K
\to G(A)$ only depends on its reduction modulo $\m_A$, that is, the
representation $\Gamma_K \to G(F)$ coming from $\rho$.
\end{rema}

\subsubsection{Rewriting of the definition}

We want to give a definition of the Hodge type that is easier to
manipulate, using the fact that $G$ is reductive.

Let $T$ be a maximal torus of $G_{\ov{E}}$. Then a conjugacy class of
cocharacters $\bG_{m,\ov{E}} \to G_{\ov{E}}$ is the same thing as an
orbit under the Weyl group in the group $X_*(T)$ of cocharacters of $T$.
Let now $\hat{G}$ be the dual group of $G_{\ov{E}}$ (it is an algebraic
group over $\ov{E}$), and $\hat{T}$ a maximal torus of $\hat{G}$ so that
we identify canonically $X_*(T)$ and $X^*(\hat{T})$, in a way that is
compatible with the action of the Weyl group. Then finally, our conjugacy
class is the same as an element of the set $X^*(\hat{T})_+$  of dominant
characters of $\hat{T}$ (for some choice of a Borel subgroup of
$\hat{G}$).

\begin{defi}
Let $\rho : \Gamma_K \to G(E)$ be a potentially semi-stable
representation. Its Hodge type is the element of $X^*(\hat{T})_+^\e$
given by the construction above.
\end{defi}

\subsubsection{Functoriality}

Let $H$ be another reductive group over $E$.
Let $u$ be a homomorphism of algebraic groups from $G$ to $H$.
We can choose maximal tori $T_G$ of $G$ and and $T_H$ of $H$ so that
$u(T_G) \subset T_H$. Then we get a map $u_* : X_*(T_G) \to X_*(T_H)$. We
can then choose maximal tori $\hat{T}_G$ of $\hat{G}$  and $\hat{T}_H$ of
$\hat{H}$ so that $u_*$ becomes a map $X^*(\hat{T}_G) \to
X^*(\hat{T}_H)$, and we can choose ordering on the cocharacter groups so
that $u_*(X^*(\hat{T}_G)_+) \subset X^*(\hat{T}_H)_+$.

We now denote by $\rho_G$ a continuous representation from $\Gamma_K$ to
$G(A)$ and $\rho_H$  the composition :

$$
u^*\rho_G = u \circ \rho_G : \Gamma_K \rightarrow H(A).
$$

\begin{prop}
\label{Hodgefunctorial}
\begin{enumerate}
\item 
If $\rho_G$ is de Rham (resp. semi-stable, crystalline), then so is $\rho_H$.

\item 
Assume $\rho_G$ is de Rham and fix $\mu$ a cocharacter of $G_{\overline{E} \otimes_{\Qp} K}$ representing the Hodge type of  $\rho_G$.
Then, the $p$-adic Hodge type of $\rho_H$ is the $H(\overline{E}
\otimes_{\Qp} K)$-conjugacy class of the cocharacter  $u \circ \mu$  of
$H_{\overline{E} \otimes_{\Qp} K}$.

\item
Assume $\rho_G$ is de Rham and let $\lambda \in X^*(\hat{T}_G)_+^\e$ be the
Hodge type of $\rho_G$. Then the Hodge type of $\rho_H$ is $u_*\lambda
\in X^*(\hat{T}_H)_+^\e$.

\end{enumerate}
\end{prop}

\begin{proof}
The first assertion comes from the characterisation of de Rham
$A$-representations of $\Gamma_K$ recalled in~\ref{sssec:dRGrepGal} and
from the fact that the functor $\mathcal{R}_{A,\rho_H}$ is the
composition $\mathcal{R}_{A,\rho_G} \circ u^*$ of $\mathcal{R}_{A,\rho_G}
$ with the functor $u^*$ from the category $\Rep_E H$ of (finite
dimensional $E$-) representations of $H$ to representations of $G$,
$\Rep_E G$ :

$$
\begin{array}{r r c l}
f^* : & \Rep_E H & \rightarrow  & \Rep_E G \\
& (W,\pi_W) & \mapsto & (W,\pi_W \circ f_E).
\end{array}
$$
The other statements follow from the construction and from Proposition
\ref{functoriality}.
\end{proof}

\subsubsection{Examples: $\GL_n$ and $\PGL_n$}
\label{exHodgetype}

As $\GL_n$ is its own dual group, a
Hodge type is an element of $X^*(T_{\GL_n})_+^\e$.

We can identify $X^*(T_{\GL_n})$ with $\Z^n$, with the Weyl group acting
as the symmetric group. So we recover that a Hodge type for $\GL_n$ is
the familiar family indexed by $\e$ of unordered $n$-uples of integers,
and the Hodge type is regular (that is, for each $\iota \in \e$, all the
weights at $\iota$ are distinct) if and only if each character is
regular.

We can identify $X^*(T_{\GL_n})_+$ with 
the set of $\ell = (\ell_1,\dots,\ell_n)$ in $\Z^n$ such that 
$\ell_{1} \geq \dots \geq \ell_{n-1} \geq \ell_n$. Then the regular characters are
those for which $\ell_{1} > \dots > \ell_{n-1} > \ell_n$. 

Let now $\lambda$ be a Hodge type for $\PGL_n$. 
The dual group of $\PGL_n$ is $\SL_n$, so a Hodge type 
is an element of $X^*(T_{\SL_n})_+^\e$. 

The group $X^*(T_{\SL_n})$ is isomorphic to $\Z^n/\Z\ul{1}$ (where
$\ul{1} = (1,\dots,1)$), with the natural map $T_{\SL_n} \to T_{\GL_n}$
inducing the natural projection $\Z^n \to \Z^n/\Z\ul{1}$. 
The regular characters in $X^*(T_{\SL_n})$ are the images of the regular
characters in $X^*(T_{\GL_n})$, and we can choose an ordering on 
$X^*(T_{\SL_n})$ so that the dominant characters are exactly the 
images of the dominant characters.

\begin{defi}
\label{standardlift}
Let $\lambda \in X^*(T_{\SL_n})_+$ and let $\ell \in X^*(T_{\GL_n})$. We
say that $\ell$ is a lift
of $\lambda$ if the image of $\ell$ by the natural map $X^*(T_{\GL_n})
\to X^*(T_{\SL_n})$ is $\lambda$. This implies that $\ell \in 
X^*(T_{\GL_n})_+$. 
\end{defi}

The functoriality of the construction translates as
follows: let $r : \Gamma_K \to \GL_n(\ov\Q_p)$ a de Rham representation,
let $\rho : \Gamma_K \to \PGL_n(\ov\Q_p)$ coming from $r$ by the
projection $\GL_n(\ov\Q_p) \to \PGL_n(\ov\Q_p)$. Then if $\ell \in
X^*(T_{\GL_n})_+^\e$ is the Hodge type of $r$, then the Hodge
type of $\rho$ is the image of $\ell$ in $X^*(T_{\SL_n})_+^\e$ via the
natural map $X^*(T_{\GL_n})_+^\e \to X^*(T_{\SL_n})_+^\e$.

Consider for example the case where $n=2$. Then a Hodge type for $\GL_2$
is a family $(a_\iota,b_\iota)_{\iota\in \e}$ of pairs of integers with
$a_\iota \geq b_\iota$. We can write a Hodge type for $\PGL_2$ as a
family of non-negative integers $(c_\iota)_{\iota\in\e}$, and a de Rham
representation with values in $\GL_n(E)$ with Hodge type
$(a_\iota,b_\iota)_{\iota\in \e}$ gives rise to a de Rham representation
with values in $\PGL_n(E)$ with Hodge type
$(a_\iota-b_\iota)_{\iota\in \e}$.

\subsection{The inertial type of a de Rham $G$-representation of $\Gamma_K$}
\label{inertialtype}

We recall the definition of the inertial type of a potentially
semi-stable representation, after Balaji (\cite[Section 2.5]{Bal}), but in
a form more suited to our needs. We then prove a functoriality property.

Fix a local Artinian finite $E$-algebra $A$. Our goal is to define the inertial
type of a potentially semi-stable representation $\rho : \Gamma_K \to
G(A)$. Let $F$ be the residue field of $A$, so that $F$ is a finite
extension of $E$. We will define the inertial type of $\rho$ as the inertial
type of the representation $\Gamma_K \to G(F)$ given by $A \to F$. So we
forget about $A$ (This coincides with Balaji's definition of the inertial
type by \cite[Prop 2.5.4]{Bal}). We fix an embedding $F \to \ov\Q_p$.

\subsubsection{Inertial type of a potentially semi-stable representation}

Let $L$ be a fixed finite, Galois extension of $K$.  Let $\cc_L$ be the
category of continuous potentially semi-stable representations of $\Gamma_K$ on a
finite-dimensional $F$-vector space, that become semi-stable on $L$.

Let $\I$ (resp. $\W$) be the category of representations of $I_K$ (resp.
$W_K$) over an $\ov\Q_p$-vector space of finite dimension with open
kernel.

By \cite[Appendix B]{CDT}, we can attach to any object $r$ of $\cc_L$ an
isomorphism class of Weil--Deligne representations. By considering only
the Weil representation part,
we can attach to it, up to isomorphism, an object of $\W$.
Restricting further to $\I$, we get an object of $\I$, which
we call the inertial type of $r$.

Denote by $L_0$ the
maximal unramified extension of $\Q_p$ contained in $L$. 

In the rest of this paragraph, we explain why the construction of the
inertial type can be made functorial.

\begin{prop}
\label{weilfunc}
Let $i: L_0 \to \ov\Q_p$ be a field homomorphism.
There is a functor $\Ww_L: \cc_L \to \W$ which is exact and
compatible with duality and tensor products, such that $\Ww_L(r)$ is
isomorphic to the Weil representation attached to $r$ for any object $r$
of $\cc_L$.
Moreover, choosing another $i : L_0 \to \ov\Q_p$ gives
an isomorphic functor.
\end{prop}

We sometimes write $\Ww_{L,i}$ instead of $\Ww_L$ if we want to
emphasize the role of the map $i: L_0 \to \ov\Q_p$.

\begin{proof}
Let $(r,V)$ be an objet of $\cc_L$.
Consider $D_{st}^L(V) =
(B_{st}\otimes_{\Q_p}(V\otimes_F\ov\Q_p))^{\Gamma_L}$. This is a free
$L_0\otimes_{\Q_p}\ov\Q_p$-module of rank $n = \dim_FV$. 
It is a filtered $(\phi,N)$-module, so
in particular endowed with a semi-linear map $\phi$, and has an action of
$\Gal(L/K)$ coming from the action of $\Gamma_K$.

We first define an $L_0\otimes_{\Q_p}\ov\Q_p$-linear action of $W_K$ on
$D_{st}^L(V)$. If $g \in \Gamma_K$, define $\alpha(g) \in \Z$ so that the
image of $g$ in $\Gal(\ov\F_p/\F_p)$ is $F_p^{\alpha(g)}$. Then $g \in W_K$
acts on $D_{st}^L(V)$ by $\ov{g}\circ \phi^{-\alpha(g)}$, where $\ov{g}$
is the image of $g$ in $\Gal(L/K)$. 

Then we define $\Ww_L(V)$ as follows: it is the representation of $W_K$
on $D_{st}^L(V)\otimes_{L_0\otimes_{\Q_p}\ov\Q_p}\ov\Q_p$, where the map
$L_0\otimes_{\Q_p}\ov\Q_p \to \ov\Q_p$ is given by $i$.

\medskip

Note that the restriction of this representation to $I_K$ (that is,
$\Ii_L(V)$) is easier to describe:
The inertia subgroup $I_K$ acts naturally on $D_{st}^L(V)$, in a way that is
$L_0\otimes_{\Q_p}\ov\Q_p$-linear, with $I_L$ acting trivially.
Then the $I_K$-representation is given by
$D_{st}^L(V)\otimes_{L_0\otimes_{\Q_p}\ov\Q_p}\ov\Q_p$, where the map
$L_0\otimes_{\Q_p}\ov\Q_p \to \ov\Q_p$ is given by $i$.

\medskip

Let now $j : L_0 \to \ov\Q_p$ be another map, and let us compare
$\Ww_{L,i}$ and $\Ww_{L,j}$. We can write $j = i \circ \Phi^m$ for some $m$,
where $\Phi$ is the Frobenius of $L_0$. Then $\Phi^m$ induces compatible
isomorphisms of $W_K$-representations
$D_{st}^L(V)\otimes_{L_0\otimes_{\Q_p}\ov\Q_p,i}\ov\Q_p
\to
D_{st}^L(V)\otimes_{L_0\otimes_{\Q_p}\ov\Q_p,j}\ov\Q_p$
for all $V$, and so gives an isomorphism between the functors
$\Ww_{L,i}$ and $\Ww_{L,j}$.
\end{proof}

\subsubsection{Definition of the inertial type of a potentially semi-stable $G$-representation}

\begin{defi}
\label{deftypeinGobj}
Let $\ul{\rho}$ be a $G$-object with values in $\cc_L$. Its inertial type
$\tau$ is the $G(\ov\Q_p)$-conjugacy class of representation $I_K \to
G(\ov\Q_p)$ with open kernel defined as follows: $\Ww_L \circ \ul{\rho}$ is a
$G$-object with values in $\W$. Then $\tau$ is the restriction to inertia
of the representation $W_K \to G(\ov\Q_p)$ attached to $\Ww_L\circ \ul{\rho}$
by Proposition \ref{Gobjtonaive}.
\end{defi}

Note that as observed in Remark \ref{remGobjtonaive}, the
representation $W_K \to G(\ov\Q_p)$ is only defined up to conjugacy by an
element of $G(\ov\Q_p)$.

\begin{defi}
\label{deftypein}
Let $\rho : \Gamma_K \to G(F)$ be a potentially semi-stable
representation. Its inertial type is the $G(\ov\Q_p)$-conjugacy class of
representation $I_K \to G(\ov\Q_p)$ with open kernel defined as follows:
let $L$ be a finite Galois extension of $K$ such that $\rho$ becomes
semi-stable on $L$.
Then $\rho$ defines a $G$-object $\ul{\rho}$ with values in $\cc_L$,
and $\tau$ is attached to $\ul{\rho}$ as in Definition
\ref{deftypeinGobj}.
\end{defi}

\begin{prop}
The definition above does not depend on the choice of $L$, or the
embedding $L_0 \to \ov\Q_p$ used to define the functor $\cc_L \to \W$.
\end{prop}

\begin{proof}
Fix first $L$, and choose two different embeddings $i,j : L_0 \to
\ov\Q_p$. Then the functors $\Ww_{L,i}$ and $\Ww_{L,j}$ from $\cc_L$ to $\W$
are isomorphic by Proposition \ref{weilfunc}, 
and so the functors $\Ww_{L,i}\circ \ul{\rho}$ and
$\Ww_{L,j}\circ \ul{\rho}$ are also isomorphic, and so give rise by the
process of Proposition \ref{Gobjtonaive} to two representations
$W_K \to G(\ov\Q_p)$ that are conjugate by an element of $G(\ov\Q_p)$.

Now fix $L$ such that $\Gamma_K$ becomes semi-stable on $L$, and $L'$ be
an extension of $L$ which is Galois over $K$. We fix $i : L_0 \to
\ov\Q_p$, and $i' : L'_0 \to \ov\Q_p$ extending $i$. Let $u : \cc_L \to
\cc_{L'}$ be the natural functor, then we check immediately
that $\Ww_{L',i'} \circ u = \Ww_{L,i}$ which gives the result.
\end{proof}

Note that by construction, a representation $\tau : I_K \to G(\ov\Q_p)$
can appear as the inertial type of a representation only if it extends to
a representation $W_K \to G(\ov\Q_p)$. Hence the following definition:

\begin{defi}
\label{definertialabs}
An inertial type for $G$ is an isomorphism class of representation $I_K
\to G(\ov\Q_p)$ with open kernel that can be extended to a representation 
$W_K \to G(\ov\Q_p)$.
\end{defi}

\subsubsection{Functoriality}

Let now $H$ be another reductive group defined over $E$, and
let $u : G \to H$ be a morphism of algebraic groups over $E$.
From Proposition \ref{functoriality} we get the following:

\begin{prop}
\label{inertialfunctorial}

Let $\rho_G : \Gamma_K \to G(F)$ be a potentially semi-stable
representation. Let $\rho_H = u_F \circ \rho : \Gamma_K \to H(F)$, so that
$\rho_H$ is also a potentially semi-stable representation.

Let $\tau_G : I_K \to G(\ov\Q_p)$ and $\tau_H : I_K \to H(\ov\Q_p)$ be the
inertial types of $\rho_G$ and $\rho_H$ respectively. Then $\tau_H =
u_{\ov\Q_p}\circ \tau_G$, up to conjugation by an element of $H(\ov\Q_p)$.
\end{prop}

\subsection{Application to potentially semi-stable deformation rings}

From now on,
$G$ is a smooth linear algebraic group defined over $\O_E$, such that
$G_E$ is reductive. The results of the previous sections then apply to
its generic fiber $G_E = G \times_{\O_E}E$.

Let $\Noe$ be the category of complete local Noetherian $\O_E$-algebras
with residue field $\F$ (i.e. such that the structural homomorphism $\O_E
\rightarrow A$ induces an isomorphism $\F \rightarrow A / \m_A$).
If $A$ is an object of $\Noe$, we denote by
$\m_A$ its maximal ideal and by $\xi_A$ the natural projection from $A$
to $A / \m_A$ ; by abuse of notation $\xi_A$ will also denote the
application from $A$ to $\F$.

\subsubsection{Universal deformation ring}
\label{univdefrings}

\begin{defi}
\label{deffunctdef}
We denote by $\Delta_G^\square({\ov{\rho}})$ the functor from $\Noe$ to the
category of sets which associates to every $A$ in $\Noe$ the sets of
continuous homomorphisms $\rho$ from $\Gamma_K$ to $G(A)$ such that
$ \xi_A \circ \rho = \ov{\rho} $.

\end{defi}

From \cite{Maz} for the case of $\GL_n$, and \cite{Til} and 
\cite[Theorem 1.2.2]{Bal} for the general case, we see that:

\begin{theo}
\label{defcadrrings}
The functor $\Delta^\square_G({\ov{\rho}})$ is 
representable by a complete local noetherian $\O_E$-algebra 
$R^\square_G(\ov{\rho})$.
\end{theo}

\subsubsection{Potentially semi-stable deformation rings}
\label{defpotrings}

If $F$ is a finite extension of $E$, the universal ring classifying
deformations to $\O_F$-algebras is
$R_G^\square(\ov\rho)\otimes_{O_E}\O_F$. We denote it by
$R_G^\square(\ov\rho)_F$.

We have the following result due to Balaji (\cite[Theorem 3.0.12]{Bal}),
which generalizes to reductive groups the results of \cite{Kis08}:

\begin{theo}
Let $\ov\rho : \Gamma_K \to G(\F)$ be a continuous Galois representation.
Let $\tau$ be an inertial type for $G$, and $v$ be a Hodge type for $G$.
Then there exist a finite extension $F$ of $E$, and 
a quotient $R(\tau,v,\ov\rho)'$ of $R_G^\square(\ov\rho)_F[1/p]$ satisfying
the following property: for any finite local $F$-algebra $B$, a map of
$F$-algebras $\zeta : R_G^\square(\ov\rho)_F[1/p] \to B$ factors through 
$R(\tau,v,\ov\rho)'$ if and only if the induced representation
$\rho_\zeta : \Gamma_K \to G(B)$ is potentially semi-stable with inertial
type $\tau$ and Hodge type $v$.
\end{theo} 

Here $F$ is taken so that $\tau$ can be defined over $F$. In order to
simply notations, in the sequel we simply enlarge $E$ so that $\tau$ can
be defined over $E$.


We define $R(\tau,v,\ov\rho)$ as the unique quotient of
$R_G^\square(\ov\rho)$ such that $R(\tau,v,\ov\rho)[1/p]$ is
$R(\tau,v,\ov\rho)'$ and $R(\tau,v,\ov\rho)$ is $p$-torsion-free.  We
summarize the properties of $R(\tau,v,\ov\rho)$, coming from Balaji's
result and \cite[Theorem 3.3.3]{BG}:

\begin{prop}
\label{proprings}
The quotient $R(\tau,v,\ov\rho)$ of $R_G^\square(\ov\rho)$ satisfies the
following property:
for any finite local $E$-algebra $B$, a map of
$E$-algebras $\zeta : R_G^\square(\ov\rho)[1/p] \to B$ factors through 
$R(\tau,v,\ov\rho)[1/p]$ if and only if the induced representation
$\rho_\zeta : \Gamma_K \to G(B)$ is potentially semi-stable with inertial
type $\tau$ and Hodge type $v$.

Moreover, $R(\tau,v,\ov\rho)$ is $p$-torsion free, hence flat over $\Z_p$, and 
$\spec R(\tau,v,\ov\rho)[1/p]$ is reduced.
\end{prop}

\subsubsection{Functorality for potentially semi-stable deformation
rings}
 
Let $H$ be another smooth reductive group defined over $\O_E$, and let $u
: G \to H$ be a homomorphism of algebraic groups over $\O_E$.

\begin{prop}
\label{functordefrings}
Let $\ov\rho : \Gamma_K \to G(\F)$ be a continuous Galois representation.
Then we have a natural morphism of functors $u_* : \Delta^\square_G(\ov\rho) \to
\Delta^\square_H(u_\F \circ \ov\rho)$, where for $A \in \Noe$ the map
$\Delta^\square_G(\ov\rho)(A) \to
\Delta^\square_H(u_\F \circ \ov\rho)(A)$ is given by $\rho \mapsto u_A
\circ \rho$.

This gives rise to a local morphism of $\O_E$-algebras $u^* : R_H^\square(u_\F
\circ \ov\rho) \to R_G^\square(\ov\rho)$.
\end{prop}

Let now $\tau : I_K \to G(\ov\Q_p)$ be an inertial type, and let $v$ be a
Hodge type for $G$. Then $u_*\tau : I_K \to H(\ov\Q_p)$ is an inertial
type for $H$, and $u_*v$ is a Hodge type for $H$ by Propositions
\ref{inertialfunctorial} and \ref{Hodgefunctorial}.

Then by the factorization property of Proposition \ref{proprings}, we
have the following result using standard arguments:

\begin{theo}
\label{functorpstdefrings}
The natural map $R_H^\square(u_\F \circ \ov\rho) \to
R_G^\square(\ov\rho) \to R(\tau,v,\ov\rho)$ 
factors through the quotient 
$R(u_*\tau,u_*v,u\circ \ov\rho)$ of $R_H^\square(u_\F \circ \ov\rho)$
and gives rise to a morphism of $\O_E$-algebras
$u^* : R(u_*\tau,u_*v,u\circ \ov\rho) \to R(\tau,v,\ov\rho)$.
\end{theo}

\section{Lifts from $\PGL_n$ to $\GL_n$}
\label{lifts}

\subsection{The lift/deformation lemma}

\begin{prop}
\label{liftdef}
Let $A$ be a complete local noetherian ring with maximal ideal $\m_A$
and finite residue field $\F$ of characteristic $p$.

Let $\Gamma$ be a profinite group. Let $\bar\rho : \Gamma \to \PGL_n(\F)$
be a continuous representation, and let $\rho : \Gamma \to \PGL_n(A)$ be
a continuous deformation of $\bar\rho$ (that is, $\rho$ modulo $\m_A$ is
$\bar\rho$).

Then $\bar\rho$ lifts to a continuous representation $\bar{r} : \Gamma
\to \GL_n(\F)$ if and only if $\rho$ lifts to a continuous representation
$r : \Gamma \to \GL_n(A)$.

Moreover, if we have a lift $\bar{r}$ of $\bar\rho$, then we can take the
lift $r$ of $\rho$ such that $r$ modulo $\m_A$ is $\bar{r}$.
\end{prop}

Before proving the Proposition, we need to state a few lemmas.

Let $\mu^{(p)}(A)$ be the subgroup of units of $A$ that are roots of
unity of order prime to $p$. Then the natural projection $\mu^{(p)}(A)
\to \F^\times$ is an isomorphism.

Let $Q_n(A)$ be the subgroup of elements of $\GL_n(A)$ with determinant
in $\mu^{(p)}(A)$.

\begin{lemm}
\label{exQ}
We have the following commutative diagrams with exact rows:
\[
\begin{CD}
1 @>>>  \mu^{(p)}(A) @>>> Q_n(A) @>>>  \PGL_n(A) @>>>  1 \\
@.       @VVV              @VVV         @VVV     @. \\
1 @>>> \F^\times @>>> \GL_n(\F) @>>> \PGL_n(\F) @>>>  1
\end{CD}
\]

where the first vertical map is an isomorphism and the other two are
surjective.
\end{lemm}

We recall the following well-known lemma:

\begin{lemm}
\label{lifth2}
Let
$$1 \to H \to G \xrightarrow{\pi} Q \to 1$$
be an exact sequence of topological groups, with $H$ central.
Let $\sigma$ be a continuous section of $\pi$ with $\sigma(1) = 1$. Let
$c \in Z^2(Q,H)$ defined by: $c(x,y) =
\sigma(x)\sigma(y)\sigma(xy)^{-1}$. The image of $c$ in $H^2(Q,H)$
does not depend on the choice of $\sigma$.

Let $\Gamma$ be a topological group, and $\rho : \Gamma \to Q$ be a
continuous homomorphism. Then $\rho$ lifts to a continuous homomorphism
$r : \Gamma \to G$ if and only if the image of $\rho_*c$ in $H^2(\Gamma,H)$
is zero.
\end{lemm}

Here all cycles are assumed to be continuous.

\begin{proof}
Assume first that $\rho_*c$ is zero in $H^2(\Gamma,H)$. This means that
there exists a map $\alpha : \Gamma \to H$ such that $d\alpha = c \circ
\rho$. Set $r(x) = \alpha(x)^{-1}\sigma(\rho(x))$.

Assume now that $\rho$ lifts to $r$. Then there exists some continuous
function $\alpha : \Gamma \to H$ such that $\sigma(\rho(x)) =
\alpha(x)r(x)$ for all $x \in \Gamma$, and we see that $d\alpha = c\circ
\rho$.
\end{proof}

Let $(A,\m_A,\F)$ be as in the statement of Proposition \ref{liftdef}.

\begin{lemm}
\label{compatlift}
Let $\bar\sigma : \PGL_n(\F) \to \GL_n(\F)$ be a section. Then there exists
a continuous section $\sigma : \PGL_n(A) \to Q_n(A)$ which is
compatible to $\bar\sigma$, that is, $\sigma(x)$ modulo $\m_A$ is
$\bar\sigma(x\text{ modulo }\m_A)$ for any $x \in \PGL_n(A)$.
\end{lemm}

\begin{proof}
It is enough to construct such a section on the subgroup $\mathcal{K}$ of
$\PGL_2(A)$ of matrices that reduce to the identity modulo $\m_A$. Note that
the map $\left(\SL_n(A) \cap (\id+\m_A M_n(A))\right) \to \mathcal{K}$ is bijective,
we take for $\sigma$ the inverse map.
\end{proof}

\begin{proof}[Proof of Proposition \ref{liftdef}]
Choose as in Lemma \ref{compatlift} two continuous sections
$\bar\sigma : \PGL_n(\F) \to \GL_n(\F)$
and
$\sigma : Q_n(A) \to \GL_n(A)$
that are compatible.
Recall from Lemma \ref{lifth2} the existence of elements $c_0 \in
Z^2(\PGL_n(A),\mu^{(p)}(A))$ and $c_p \in Z^2(\PGL_n(\F),\F^\times)$
constructed from these sections.

Let $\gamma_0 = \rho_*c_0 \in Z^2(\Gamma,\mu^{(p)}(A))$ and
$\gamma_p = \bar\rho_*c_p \in Z^2(\Gamma,\F^\times)$. Then by Lemma
\ref{compatlift}, they correspond to each other via the isomorphism
$\mu^{(p)}(A) \to \F^\times$.

This gives the equivalence of the liftability of $\bar\rho$ and $\rho$.

Now we need to show that for a given continuous lift $\bar{r}$ of
$\bar\rho$, there exists a continuous lift $r : \Gamma \to Q_n(A)$ of
$\rho$ which is compatible to $\bar{r}$.

By the proof of Lemma
\ref{lifth2}, the datum of $\bar{r}$ gives some continuous $\alpha_p : \Gamma \to
\F^\times$ such that $c_p = d\alpha_p$.  Define $\alpha_0 : \Gamma \to
\mu^{(p)}(A)$ via the isomorphism $\mu^{(p)}(A) \to \F^\times$, then $c_0
= d\alpha_0$, and using $\alpha_0$ we can define a lift $r$ of $\rho$
which is compatible to $\bar{r}$.
\end{proof}

\subsection{Lifts of Galois representations}

We recall the following theorem of Tate and its corollary (see
\cite[Paragraph 6]{Se77}), where $F$ is a local or a global field:

\begin{theo}
\label{Tate}
The group $H^2(\Gal(\ov{F}/F),\Q/\Z)$ is zero, and
in particular for any prime $\ell$ we have
$H^2(\Gal(\ov{F}/F),(\Q/\Z)[\ell^\infty]) = 0$.  
\end{theo}

\begin{coro}
\label{liftexists}
Any continuous representation of $\Gal(\ov{F}/F)$ taking its values
in $\PGL_n(\F)$ for some finite field $\F$ lifts to a continuous
representation with values in $\GL_n(\F')$ for some finite extension
$\F'$ of $\F$.
\end{coro}

\subsection{Lifts of potentially semi-stable representations}

\subsubsection{Notation}
\label{not}

Let $\e$ be the set of embeddings $K \to \ov\Q_p$.
Fix a uniformizer $\varpi_K$ of $\O_K$.
Let $\chi_K : \Gamma \to \O_K^\times$ be the Lubin-Tate character
corresponding to $\varpi_K$ (the restriction of $\chi_K$ to
inertia does not depend on the choice of $\varpi_K$).
For $\iota\in \e$, let $\chi_\iota =
\iota\circ\chi_K$. Then $\chi_\iota$ is a crystalline character, with
all its Hodge-Tate weights equal to $0$ except the one at $\iota$ which
is $1$. Note that $\cyc$ and $\prod_{\iota\in\e}\chi_\iota$ differ by an
unramified character.

\subsubsection{Existence of potentially semi-stable lifts}
\label{existence}

\begin{prop}
\label{potstliftexists}
Let $E$ be a finite extension of $\Q_p$.
Let $\rho : \Gamma_K \to \PGL_n(\O_E)$ be a potentially semi-stable
representation.

Then there exist a finite extension $E'$ of
$E$ and a potentially semi-stable lift $r : \Gamma_K \to \GL_n(\O_{E'})$
of $\rho$.

Moreover, if $\rho$ is potentially crystalline then we can choose $r$ to
be also potentially crystalline.
\end{prop}

\begin{proof}
Write $V = \O_E^n$, so that $\rho$ is a map $\Gamma_K \to \PGL(V)$.
Let $r$ be any lift of $\rho$ to a continuous representation
$r : \Gamma_K \to \GL(V)$ (which exists after a finite extension of scalars
from $\O_E$ to some $\O_{E'}$ by Corollary \ref{liftexists} and
Proposition \ref{liftdef}).

Let $\sigma: \Gamma_K \to \GL(V^*\otimes V)$ be given by the adjoint
representation of $r$. Then $\sigma$ factors through $\rho$. This means
that $\sigma = r^* \otimes r$ is a potentially semi-stable
representation, as $\rho$ is, and the map $\PGL(V) \to \GL(V^*\otimes V)$
is algebraic. 

By \cite[Corollary 2.3.3]{DiM},
as $r \otimes r^*$ is de Rham, then there exists a twist of
$r$ by some character $\alpha$ that is de Rham. So $r\otimes\alpha$ gives us a lift
of $\rho$ that is de Rham.

Assume now that $\rho$ is actually potentially crystalline. Considering the
restriction of $\rho$ and $r\otimes\alpha$ to some finite extension of $K$,
and using \cite[Corollary 3.2.2]{DiM}, we see that $r\otimes\alpha$ is
also potentially crystalline.
\end{proof}

The functoriality result of Proposition \ref{inertialfunctorial}
translates in this situation as:

\begin{lemm}
\label{typequotient}
Let $\rho : \Gamma_K \to \PGL_n(E)$ be a potentially semi-stable
representation with Hodge type $\lambda$
and inertial type $\tau$.
Let $r : \Gamma_K \to \GL_n(E)$ be a potentially semi-stable
representation with Hodge type $\ell$ and
inertial type $t$, such that $u_*r = \rho$. Then $u_*t =
\tau$ and $\ell$ is a lift of $\lambda$.
\end{lemm}

So we deduce:

\begin{coro}
\label{inertiallift}
Let $\rho$ be as above, and let $\tau$ be its inertial type. Then $\tau$
lifts to an inertial type $t$ for $\GL_n$.
\end{coro}

\subsubsection{Determinant conditions}
\label{detcond}

Let $\ell$ be a Hodge type for $\GL_n$. Then $\ell =
(\ell_\iota)_{\iota\in\e}$ with each $\ell_\iota$ of the form
$(\ell_{1,\iota}\dots,\ell_{n,\iota}) \in \Z^n$.
We define $\Sigma_\iota(\ell) \in \Z$ to be $\sum_{i=1}^n\ell_{i,\iota}$.

\begin{defi}
\label{deficompatgl}
Let $\ell$ be a Hodge type for $\GL_n$, $t$ an inertial type for $\GL_n$. 

Let $\psi :
\Gamma_K \to \ov\Q_p^\times$ be a character. We say that $\psi$ is compatible
to $\ell$ and $t$ if the restriction of $\psi$
to inertia is $\det(t)\prod\chi_\iota^{\Sigma_\iota(\ell)}$.

Let $\psi :
\Gamma_K \to \ov\F_p^\times$ be a character. We say that $\psi$ is compatible
to $\ell$ and $t$ if the restriction of $\psi$
to inertia is $\ov{\det(t)\prod\chi_\iota^{\Sigma_\iota(\ell)}}$.
\end{defi}

We see easily that if there exists a potentially semi-stable
representation $r : \Gamma_K \to \GL_n(\ov\Q_p)$ that has Hodge type $\ell$,
inertial type $t$, and determinant $\psi$, then $\psi$ is compatible to
$\ell$ and $t$.

We now want to describe a compatibility condition for representations
with values in $\PGL_n(\ov\Q_p)$.

Let $A$ be either $\ov\F_p$ or $\ov\Q_p$. We denote by by $\cc(A)$ the group of
continuous characters $\Gamma \to A^\times$, and by $\cc'(A)$ the quotient
$\cc(A)/\cc(A)^n$. Let $\rho : \Gamma_K \to \PGL_n(A)$ be a continuous
representation. Then we can attach to $\rho$ a well-defined element
$\det(\rho) \in \cc'(A)$. Indeed, fix any continuous lift $r : \Gamma_K \to
\GL_n(A)$ of $\rho$ (this exists thanks to Corollary \ref{liftexists} if
$A = \ov\F_p$, or a combination of Corollary \ref{liftexists} and
Proposition \ref{liftdef}
if $A = \ov\Q_p$), and take for $\det(\rho)$ the class of $\det(r)$.
Note that $\det(\rho)$ depends only on the restriction of $\rho$ to
inertia, as any unramified character is in $\cc(A)^n$.

If $\lambda$ is a Hodge type for $\PGL_n$ and $\iota\in \e$, we define
$\Sigma_\iota(\lambda) \in \Z/n\Z$ as follows: let $\ell$ be a lift of
$\lambda$, set $\Sigma_\iota(\lambda)$ to be the image of
$\Sigma_\iota(\ell)$ in $\Z/n\Z$.

\begin{defi}
\label{deficompatpgl}
Let $\lambda$ be a Hodge type for $\PGL_n$, $\tau$ an inertial type. 

Let $\psi :
\Gamma_K \to \ov\Q_p^\times$ be a character. We say that $\psi$ is compatible
to $\lambda$ and $\tau$ if $\psi$ has the same image in $\cc'(\ov\Q_p)$ as
$\det(\tau)\prod\chi_\iota^{\Sigma_\iota(\lambda)}$.

Let $\psi :
\Gamma_K \to \ov\F_p^\times$ be a character. We say that $\psi$ is compatible
to $\lambda$ and $\tau$ if $\psi$ has the same image in $\cc'(\ov\F_p)$ as
$\ov{\det(\tau)\prod\chi_\iota^{\Sigma_\iota(\lambda)}}$.
\end{defi}

Then we get from Proposition \ref{potstliftexists} and Corollary
\ref{inertiallift}:

\begin{lemm}
\label{liftdetpotst}
Let $\rho : \Gamma_K \to \PGL_n(\ov\Q_p)$ be potentially semi-stable with Hodge
type $\lambda$ and inertial type $\tau$. Then $\det(\rho)$ is compatible
to $\lambda$ and $\tau$.

Let $\ov\rho : \Gamma_K \to \PGL_n(\ov\F_p)$. If $\ov\rho$ has a deformation which is
potentially semi-stable with Hodge type $\lambda$
and inertial type $\tau$ 
then $\det(\ov\rho)$ is compatible to $\lambda$ and $\tau$.
\end{lemm}

\subsubsection{Twists of inertial types}
\label{twists}

Let $\Gg_\Gamma$ be the finite group of continuous characters $\Gamma_K \to
\ov\Q_p^\times$ of order dividing $n$. Reduction modulo $p$
induces a bijection between $\Gg_\Gamma$ and the group of continuous
characters $\Gamma_K \to \ov\F_p^\times$ of order dividing $n$, as $p$ is
prime to $n$. 

Let $\Gg_I$ be the finite group of continuous characters $I_K \to
\ov\Q_p^\times$ of order dividing $n$. Then the
kernel of the natural restriction map $\Gg_\Gamma \to \Gg_I$ is of order
$n$, generated by $\unr(\zeta)$ for some primitive $n$-th root of unity
$\zeta$. The cardinality of $\Gg_I$ is $\gcd(n,q-1)$: indeed, any element
of $\Gg_I$ is trivial on the wild inertia subgroup, so we are counting
characters from a cyclic group of order $q-1$ to a cyclic group of order
$n$.

Let $\tau : I_K \to \PGL_n(\ov\Q_p)$ be an inertial type for $\PGL_n$.
Let $t$ be a lift of $\tau$ to an inertial type for $\GL_n$ (that exists
thanks to Corollary \ref{inertiallift}). We denote by
$\Gg_I(\tau)$ the group of characters $\alpha : I_K \to \ov\Q_p^\times$
such that $t$ is isomorphic to $t \otimes \alpha$.  This is a subgroup of
$\Gg_I$ and it depends only on $\tau$ and not on the choice of $t$.  Let
$\Gg_\Gamma(\tau)$ be the inverse image of $\Gg_I(\tau)$ in $\Gg_\Gamma$.

\subsubsection{Potentially semi-stable lifts}

The goal of this section is to prove the following result, which is a
refinement of Proposition \ref{potstliftexists}:

\begin{theo}
\label{liftpotst}
Let $E$ be a finite extension of $\Q_p$.
Let $\rho : \Gamma_K \to \PGL_n(\O_E)$ be a potentially semi-stable
representation with Hodge type $\lambda$ and inertial type $\tau$.  

Let $\ell$ be a lift of $\lambda$ and $t$ be a lift of $\tau$.
Let $\psi$ be a character that is compatible to
$\ell$ and $t$.

Then there exist a finite extension $E'$ of
$E$ and a potentially semi-stable lift $r : \Gamma_K \to \GL_n(\O_{E'})$ with
Hodge type $\ell$, inertial type $t$ and determinant $\psi$.  

Any lift of $\rho$ is a twist of $r$, and the set of all lifts satisfying
these properties is the set of $r \otimes \alpha$, $\alpha \in
\Gg_\Gamma(\tau)$.

The same result holds with "potentially semi-stable" replaced everywhere by
"potentially crystalline".
\end{theo}

\begin{proof}
Fix $\rho$, $\lambda$, $\ell$, $\tau$, $t$ and $\psi$ as in the statement
of the theorem.

Fix a finite extension $E'$ of $E$, and a lift $r : \Gamma_K \to \GL_n(\O_{E'})$
of $\rho$ that is potentially semi-stable (and potentially crystalline if
$\rho$ itself is potentially crystalline), which exists thanks to
Proposition \ref{potstliftexists}.

The Hodge type of $r$ is a lift of $\lambda$, so it is of the form
$\ell+a$ where $a$ is the Hodge type
$(a_\iota,\dots,a_\iota)_{\iota\in\e}$
for some integers $a_\iota$. By twisting $r$ by suitable powers of the
characters $\chi_\iota$ defined in Paragraph \ref{not}, we can assume
that $r$ has Hodge type $\ell$.

Let $t'$ be the inertial type of $r$. Then $u_*t' = \tau$ by Lemma
\ref{typequotient}, so $t' = t \otimes \alpha$ for some potentially
unramified character $\alpha$. We can twist $r$ by $\alpha^{-1}$, and so
we can assume that $r$ has inertial type $t$.

We see now that $\det(r)$ and $\psi$ have the same restriction to
inertia. So we can twist $r$ by an unramified character (up to extending
$E'$) so that $\det(r) = \psi$.

This gives the existence part of the theorem. Let now $r'$ be another
lift with the same properties. Then $r' = r \otimes \alpha$ for some
character $\alpha$ with $\alpha^n = 1$, as $r$ and $r'$ have the same
determinant. Moreover, $\alpha$ is potentially crystalline with
all its Hodge-Tate weights $0$, so it is potentially unramified, and has open
kernel. So the inertial type of $r'$ is $t\otimes\alpha$. This means that
$t$ is isomorphic to $t\otimes\alpha$. So this means that $\alpha \in
\Gg_\Gamma(\tau)$. 

On the other hand, if $\alpha \in \Gg_\Gamma(\tau)$ and $r$ is a twist
with the properties of the theorem, then so is $r \otimes \alpha$.
\end{proof}

\section{Potentially semi-stable deformation rings for $\PGL_n$}
\label{potringspgln}

Our goal in this section is to compare the potentially semi-stable deformation
rings for $\GL_n$ and $\PGL_n$.

Let $\ov{r}$ be a continuous homomorphism from $\Gamma_K$ to $\GL_n(\F)$
and $\psi$ a continuous character from $\Gamma_K$ to $\O_E^\times$ such
that $\ov{\psi} = \det \ov{r}$ where $\ov{\psi} = \xi_E \circ \psi$. 

Let $\ov{\rho} = u_*\ov{r}$, so that it is a continuous homomorphism from
$\Gamma_K$ to $\PGL_n(\F)$.

\subsection{Universal deformation rings}

In order to lighten the notations, we write 
$\Delta^\square(\ov{\rho})$ for the functor that is denoted by 
$\Delta^\square_{\PGL_n}(\ov{\rho})$ in Definition \ref{deffunctdef}, and 
$D^\square(\ov{r})$ 
for the functor that is denoted by 
$\Delta^\square_{\GL_n}(\ov{\rho})$ in Definition \ref{deffunctdef}, as
from now on we only work with the two groups $\GL_n$ and $\PGL_n$.

We also denote by $D^{\square,\psi}({\ov{r}})$ the functor from $\Noe$ to the
category of sets which associates to every $A$ in $\Noe$ the sets of
continuous homomorphisms $r$ from $\Gamma_K$ to $\GL_n(A)$ such that
$\det r = \psi$ and $ \xi_A \circ r= \ov{r} $. We have a natural morphism
of functors $D^{\square,\psi}({\ov{r}}) \to D^\square(\ov{r})$ which is
given by forgetting the condition on the determinant.

\begin{theo}
\label{compdefcadr}
The morphism of functors $D^{\square,\psi}({\ov{r}}) \to
\Delta^\square({\ov{\rho}})$ given on $D^{\square,\psi}(\ov{r})(A)$ by $r \mapsto
u_A \circ r$, is an isomorphism of functors.
\end{theo}

\begin{proof}
It is enough to show that for any $A$ the map
$D^{\square,\psi}({\ov{r}})(A) \to \Delta^\square({\ov{\rho}})(A)$ is
a bijection.

Let $\rho$  be in $\Delta^\square({\ov{\rho}})(A)$. Then there exists a
lift $r$ of $\rho$ such that $\xi_A(r) = \ov{r}$. Then $(\det r)/\psi$ is a
character taking its values in $1+\m_A$ so it has a unique $n$-th root $\phi$
taking its values in $1+\m_A$. Then $r\otimes\phi^{-1}$ is an element of
$D^{\square,\psi}({\ov{r}})(A)$ and is a lift of $\rho$.  Hence the
surjectivity.

Let now $r_1$ and $r_2$ be elements of $D^{\square,\psi}({\ov{r}})(A)$
that lift the same $\rho$. Then $r_1 = r_2\otimes\alpha$ for some
character $\alpha$ satisfying $\alpha^n = 1$ and $\ov\alpha = 1$, so
$\alpha = 1$. Hence the injectivity.
\end{proof}

We denote by $R^\square(\ov\rho)$ the ring representing the functor
$\Delta^\square(\ov\rho)$ 
(this corresponds to the ring
$R^\square_{\PGL_n}(\ov\rho)$ of Paragraph \ref{univdefrings}), 
and we
denote by
$\widetilde{R}^{\square}(\ov{r})$ the ring representing the functor 
$D^\square(\ov{r})$ (this corresponds to the ring
$R^\square_{\GL_n}(\ov{r})$ of Paragraph \ref{univdefrings}).
We also denote by $\widetilde{R}^{\square,\psi}(\ov{r})$ the quotient of 
$\widetilde{R}^{\square}(\ov{r})$ representing the functor
$D^{\square,\psi}({\ov{r}})$. So we get natural map
$R^\square(\ov{\rho}) 
\to
\widetilde{R}^{\square}(\ov{r})
\to
\widetilde{R}^{\square,\psi}(\ov{r})$.

\begin{theo}
\label{compdefcadrrings}
The ring homomorphism $R^\square(\ov{\rho}) \to
\widetilde{R}^{\square,\psi}(\ov{r})$, which corresponds to
the morphism of functors of Theorem \ref{compdefcadr},
is an isomorphism.
\end{theo}

\subsection{Comparison of potentially semi-stable deformation rings}

\subsubsection{Definition}

We fix $\ov\rho : \Gamma_K \to \PGL_n(\F)$ a continuous representation,
$\lambda$ a regular Hodge type for $\PGL_n$ (see Paragraph
\ref{exHodgetype})
and $\tau: I_K \to \PGL_n(\ov\Q_p)$ an inertial type.

We denote by $R^{\square}(\lambda,\tau,\ov\rho)$ the quotient of
$R^{\square}(\ov\rho)$ classifying the framed deformations of $\ov\rho$
that are potentially semi-stable with Hodge type $\lambda$ and
inertial type $\tau$, as in Paragraph \ref{defpotrings}.

Let $\ov{r} : \Gamma_K \to \GL_n(\F)$ be a lift of $\ov\rho$, $\psi$ a
character, 
$\ell$ a Hodge type for $\GL_n$, and
$t : I_K \to \GL_n(\ov\Q_p)$ an inertial type.
We denote by
$\tilde{R}^{\square,\psi}(\ell,t,\ov{r})$ the quotient of
$\tilde{R}^{\square,\psi}(\ov{r})$ classifying the framed deformations of
$\ov{r}$ that have determinant $\psi$, and are potentially semi-stable with
Hodge type $\ell$ and inertial type $t$.

Our goal is to understand the ring $R^{\square}(\lambda,\tau,\bar\rho)$ in terms
of rings of the form $\tilde{R}^{\square,\psi}(\ell,t,\ov{r})$.

From now on, 
we choose for $t$ a lift of  $\tau$,
and we take for 
$\ell$ a lift of $\lambda$.
We write $\ell$ as
$(\ell_\iota)_{\iota\in \e}$, with $\ell_\iota =
(\ell_{1,\iota},\dots,\ell_{n,\iota})$.

We assume
that $\det(\bar\rho)$ is compatible to $\lambda$ and $\tau$ (see Definition
\ref{deficompatpgl}).
Otherwise there is no framed deformation of
$\bar\rho$ that is potentially semi-stable with Hodge type $\lambda$
and inertial type $\tau$, and also no
framed deformation of a lift of $\bar\rho$ that is potentially
semi-stable with Hodge type lifting $\lambda$
and inertial type lifting $\tau$, so all the rings we want to consider
are zero.

We fix
$\psi : \Gamma_K \to \O_E^\times$ that is compatible to $\ell$ and $t$. 
Let $\ov{r_0}$ be a lift of $\ov\rho$ such that
$\det(\ov{r_0}) = \ov\psi$, which exists due to the compatibility
conditions.

\subsubsection{Maps between the deformation rings}

By Theorem
\ref{compdefcadrrings}, we have a natural isomorphism $c({\ov{r_0}}) :
R^{\square}(\ov\rho) \to \tilde{R}^{\square,\psi}(\ov{r_0})$.

We denote by $I(\ell,t,\ov{r_0})$ the kernel of the surjective map
$R^{\square}(\ov\rho) \to \tilde{R}^{\square,\psi}(\ov{r_0}) \to
\tilde{R}^{\square,\psi}(\ell,t,\ov{r_0})$, and $J(\lambda,\tau,\ov\rho)$
the kernel of the surjective map $R^{\square}(\ov\rho) \to
R^{\square}(\lambda,\tau,\ov\rho)$.

\begin{lemm}
\label{factorsthrough}
The map $R^{\square}(\ov\rho) \to \tilde{R}^{\square,\psi}(\ell,t,\ov{r_0})$
factors through $R^{\square}(\lambda,\tau,\ov\rho)$, that is, $J(\lambda,\tau,\ov\rho)
\subset I(\ell,t,\ov{r_0})$.
\end{lemm}

We denote by $c(\ell,t,\ov{r_0})$ the map
$R^{\square}(\lambda,\tau,\ov\rho) \to 
\tilde{R}^{\square,\psi}(\ell,t,\ov{r_0})$ given by this lemma.

\begin{proof}
The map $R^{\square}(\ov\rho) \to \tilde{R}^{\square,\psi}(\ov{r_0})$
factors through $\tilde{R}^{\square}(\ov{r_0})$, and the map
$\tilde{R}^{\square}(\ov{r_0}) \to \tilde{R}^{\square,\psi}(\ell,t,\ov{r_0})$ 
factors through $\tilde{R}^{\square}(\ell,t,\ov{r_0})$ by construction
of the potentially semi-stable deformation rings.
Then the map $R^{\square}(\ov\rho) \to \tilde{R}^{\square}(\ell,t,\ov{r_0})$
factors through $R^{\square}(\lambda,\tau,\ov\rho)$
by Theorem \ref{functorpstdefrings}.
\end{proof}

\subsubsection{Comparison of the deformation rings}
\label{defrings}

Let $\Qq(\tau) = \Gg_\Gamma/\Gg_\Gamma(\tau)$. Then for all $\alpha\in
\Qq(\tau)$ and any inertial type $t$ lifting $\tau$, the isomorphism
class of $t \otimes \alpha$ is well-defined.

We study first what happens after inverting $p$.

\begin{theo}
\label{compdefrings0}
As ideals of $R^\square(\ov\rho)[1/p]$, we have
$$
J(\lambda,\tau,\ov\rho)[1/p] 
= \prod_{\alpha\in \Qq(\tau)}
I(\ell,t\otimes\alpha,\ov{r_0})[1/p] 
$$
and the ideals $ I(\ell,t\otimes\alpha,\ov{r_0})[1/p]$ are pairwise
coprime.

As a consequence,
$$
R^{\square}(\lambda,\tau,\ov\rho)[1/p] =
\prod_{\alpha\in \Qq(\tau)}
\tilde{R}^{\square,\psi}(\ell,t\otimes\alpha,\ov{r_0})[1/p]
$$
via the map $\prod_{\alpha\in\Qq(\tau)}c(\ell,t\otimes\alpha,\ov{r_0})$.
\end{theo}

\begin{lemm}
\label{ssch}
Let $A$ be a Jacobson ring.
Let $I_1,\dots, I_r$ and $J$ be ideals of $A$, $B_i = A/I_i$ for all $i$
and $B = A/J$. Write $X = \spec(A)$, $Z_i = \spec(B_i)$, $Y = \spec(B)$
so that $Y$ and the $Z_i$ are closed subschemes of $X$. Assume that:

\begin{enumerate}
\item
$\spm(B)$ is the disjoint union of the $\spm(B_i)$.

\item
the ideals $I_1, \dots, I_r$, and $J$ are radical.

\end{enumerate}
Then $A = I_i + I_j$ for any $i \neq j$, $J = \cap_{i=1}^rI_i$, and the
natural map $B \to \prod_{i=1}^rB_i$ is an isomorphism.
\end{lemm}

\begin{proof}
We show first that $I_i + I_j = A$. If this is not the case, $I_i + I_j$
is contained in some maximal ideal $\m$, which contradicts the fact that
the maximal spectra of $B_i$ and $B_j$ are disjoint. We deduce that the
$Z_i$ are disjoint closed subschemes of $X$.

From the Jacobson property of $A$ we deduce that
$|Y|$ is the disjoint union of the $|Z_i|$. As $Y$ and the $Z_i$
are reduced subschemes of $X$, they are entirely determined by their
underlying closed subsets of $|X|$. So $J = \cap_{i=1}^rI_i$, and 
$B = \prod_{i=1}^rB_i$.
\end{proof}

\begin{proof}[Proof of Theorem \ref{compdefrings0}]
We want to apply Lemma \ref{ssch} to $A = R(\ov\rho)[1/p]$, and
$I_\alpha = I(\ell,t\otimes\alpha,\ov{r_0})[1/p]$ for $\alpha \in \Qq(\tau)$,
so that $B_\alpha = 
\tilde{R}^{\square,\psi}(\ell,t\otimes\alpha,\ov{r_0})[1/p]$,
and
$J = J(\lambda,\tau,\ov\rho)[1/p]$ so that $B = 
R^{\square}(\lambda,\tau,\ov\rho)[1/p]$.

Observe first that $R(\ov\rho)[1/p]$ is a Jacobson ring (see for example
\cite[Proposition 2.17]{Kap}).

The condition that the ideals are radical comes from the properties of
the semistable deformation rings: their generic fiber is reduced, as
recall in Proposition \ref{proprings} for the deformation rings without
the determinant condition ;  for the deformation rings of $\GL_n$ with
the determinant condition, the result follows from the comparison between
deformation rings with and without determinant condition of \cite[Lemma
4.3.1]{EG14}.

Let us show the condition on the maximal spectra.
Consider a maximal ideal of
$\tilde{R}^{\square,\psi}(\ell,t\otimes\alpha,\ov{r_0})[1/p]$. This corresponds to a map
$\tilde{R}^{\square,\psi}(\ell,t\otimes\alpha,\ov{r_0})[1/p] \to F$ for some finite
extension $F$ of $E$, and so to a representation
$r : \Gamma_K \to \GL_n(\O_F)$ that is potentially semi-stable with Hodge
type $\ell$, inertial type $t\otimes\alpha$, and $\ov{r} = \ov{r_0}$. So
this gives rise to a potentially semi-stable representation with Hodge
type $\lambda$, inertial type $\tau$, and values in $\PGL_n(\O_F)$
lifting $\ov\rho$, so a maximal ideal of
$R^{\square}(\lambda,\tau,\ov\rho)[1/p]$.

So we have that
$\cup_{\alpha\in\Qq(\tau)}
\spm(\tilde{R}^{\square,\psi}(\ell,t\otimes\alpha,\ov{r_0})[1/p])
\subset
\spm(R^{\square}(\lambda,\tau,\ov\rho)[1/p])$.

Conversely, consider $R^{\square}(\lambda,\tau,\ov\rho)[1/p] \to F$. This
defines a potentially semi-stable representation $\rho : \Gamma \to
\PGL_n(\O_F)$ with Hodge type $\lambda$ and inertial type $\tau$ deforming
$\ov\rho$. Using Theorem \ref{liftpotst}, we see that there exists a lift
$r : \Gamma \to \GL_n(\O_F)$ (after extending $F$ if necessary) that is
potentially semi-stable, with Hodge type $\lambda$, inertial
type $t$, and determinant $\psi$.

Consider $\ov{r}$. Then it is a twist of $\ov{r_0}$, with
the same determinant. It means that $\ov{r}\otimes\ov\alpha=\ov{r_0}$
for some $\alpha\in \Gg_\Gamma$. 

Replacing $r$ by $r' = r\otimes\alpha$, we get a representation $r'$
lifting $\rho$ with
$\ov{r'} = \ov{r_0}$, and $r'$ is potentially semi-stable with Hodge type
$\ell$, inertial type $t\otimes\alpha$, and determinant
$\psi$. So this means that the original point of 
$\spm(R^{\square}(\lambda,\tau,\ov\rho)[1/p])$
is in 
$\spm(\tilde{R}^{\square,\psi}(\ell,t\otimes\alpha,\ov{r_0})[1/p])$.

In conclusion, we have shown that
$$
\cup_{\alpha\in\Qq(\tau)}
\spm(\tilde{R}^{\square,\psi}(\ell,t\otimes\alpha,\ov{r_0})[1/p])
=
\spm(R^{\square}(\lambda,\tau,\ov\rho)[1/p]).
$$

Finally, the different 
$\spm(\tilde{R}^{\square,\psi}(\ell,t\otimes\alpha,\ov{r_0})[1/p])$ 
for $\alpha \in \Qq(\tau)$, are disjoint, as 
$t \otimes \alpha$ and $t \otimes \alpha'$ are non-isomorphic if $\alpha$
and $\alpha'$ are not equal in $\Qq(\tau)$.
\end{proof}

\begin{theo}
\label{compdefrings}
The map 
$$
\prod_{\alpha\in\Qq(\tau)}c(\ell,t\otimes\alpha,\ov{r_0}) :
R^{\square}(\lambda,\tau,\ov\rho) \to 
\prod_{\alpha\in\Qq(\tau)}\tilde{R}^{\square,\psi}(\ell,t\otimes\alpha,\ov{r_0})
$$
makes
$\prod_{\alpha\in\Qq(\tau)}\tilde{R}^{\square,\psi}(\ell,t\otimes\alpha,\ov{r_0})$
into a $R^{\square}(\lambda,\tau,\ov\rho)$-module that is $p$-torsion
free and generically free of rank $1$.
\end{theo}

\begin{proof}
All the rings $\tilde{R}^{\square,\psi}(\ell,t\otimes\alpha,\ov{r_0})$
and $R^{\square}(\lambda,\tau,\ov\rho)$ are $p$-torsion free, as recalled
in Proposition \ref{proprings}.
The ring $R^{\square}(\lambda,\tau,\ov\rho)$ is then flat over $\Z_p$, so
all its generic points are in characteristic $0$. So Theorem
\ref{compdefrings0} gives the last part of the statement.
\end{proof}

\subsubsection{The case of potentially crystalline deformation rings}

We can also define rings $\tilde{R}^{\square,cr,\psi}(\ell,t,\ov{r})$ and
$R^{\square,cr}(\lambda,\tau,\ov\rho)$ that classify representations that
are potentially crystalline with Hodge type $\ell$ (resp.
$\lambda$) and inertial type $t$ (resp. $\tau$). The same proofs as before give:

\begin{theo}
\label{compdefringscr}
The map 
$$
\prod_{\alpha\in\Qq(\tau)}c(\ell,t\otimes\alpha,\ov{r_0}) :
R^{\square,cr}(\lambda,\tau,\ov\rho) \to 
\prod_{\alpha\in\Qq(\tau)}\tilde{R}^{\square,cr,\psi}(\ell,t\otimes\alpha,\ov{r_0})
$$
induces an isomorphism
$$
\prod_{\alpha\in\Qq(\tau)}c(\ell,t\otimes\alpha,\ov{r_0}) :
R^{\square,cr}(\lambda,\tau,\ov\rho)[1/p] \to 
\prod_{\alpha\in\Qq(\tau)}\tilde{R}^{\square,cr,\psi}(\ell,t\otimes\alpha,\ov{r_0})[1/p]
$$
and makes
$\prod_{\alpha\in\Qq(\tau)}\tilde{R}^{\square,cr,\psi}(\ell,t\otimes\alpha,\ov{r_0})$
into a $R^{\square,cr}(\lambda,\tau,\ov\rho)$-module that is $p$-torsion
free and generically free of rank $1$.
\end{theo}

\subsection{Application to cycles}
\label{cycles}

\subsubsection{Definitions related to cycles}

We recall some definitions from \cite[Section 2.2]{EG14}.

Let $X$ be a Noetherian scheme. Let $X_d$ be the set of points of $X$ of
dimension $d$ (that is, the set of points $x$ such that $\ov{\{x\}}$ is
of dimension $d$).

\begin{defi}
A $d$-dimensional cycle on $X$ is an element of
$\Z[X_d]$.
\end{defi}

We denote by $Z_d$ the group of $d$-dimensional cycles.

The support of a cycle $\sum n_x[x]$ is the set of points $x$ such that
$n_x \neq 0$.

Let $Y$ be a closed subscheme of $X$ that is equidimensional of dimension
$d$. We define a $d$-dimensional cycle $Z(Y)$ as follows: write first $Z(Y) =
\sum_{i=1}^rZ(Y_i)$ where the $Y_i$ are the irreducible components of $Y$.
If $Y$ is irreducible, let $y$ be its generic point (so that $y \in
X_d$). Then we define $Z(Y)$ to be $m_{Y,y}[y]$ where $m_{Y,y} \in \Z$ is
the length of $\O_{Y,y}$ as an $\O_{Y,y}$-module.

\subsubsection{The $\GL_n$ case}

Let $\ov{r} : \Gamma_K \to \GL_n(\F)$ be a continuous representation,
let $\psi$ be a continuous character lifting $\det\ov{r}$, and let
$\tilde{R}^{\square,\psi}(\ov{r})$ be the ring classifying framed
deformations of $\ov{r}$ with determinant $\psi$. Let
$\tilde{\X}^\psi(\ov{r}) = \spec(\tilde{R}^{\square,\psi}(\ov{r}))$
and let
$\tilde{X}^\psi(\ov{r}) = \spec(\tilde{R}^{\square,\psi}(\ov{r})/\varpi)$
be its special fiber.

Note that if $\psi'$ is another character lifting $\det\ov{r}$ then
$\tilde{R}^{\square,\psi}(\ov{r})$
and
$\tilde{R}^{\square,\psi'}(\ov{r})$
are canonically isomorphic,
the isomorphism being given on representations by torsion by the unique
$n$-th root of $\psi/\psi'$ that reduces modulo $\varpi$ to the trivial
character.

We want to identify all the schemes 
$\tilde{\X}^\psi(\ov{r})$ and also all 
their special fibers 
$\tilde{X}^\psi(\ov{r})$.
So we choose any character $\psi$ lifting $\det\ov{r}$, and we set
$\hat{R}(\ov{r}) = \tilde{R}^{\square,\psi}(\ov{r})$, so that for any
$\psi'$ lifting $\det\ov{r}$, $\tilde{R}^{\square,\psi'}(\ov{r})$ is
canonically isomorphic to $\hat{R}(\ov{r})$ and we can forget about our
choice of $\psi$.
Set also $\tilde{\X}(\ov{r})
= \spec \hat{R}(\ov{r})$, so that each 
$\tilde{\X}^\psi(\ov{r})$ is canonically identified to 
$\tilde{\X}(\ov{r})$ and
each 
$\tilde{X}^\psi(\ov{r})$ is canonically identified to the special fiber
$\tilde{X}(\ov{r})$ of 
$\tilde{\X}(\ov{r})$.

Let $d = n^2 + [K:\Q_p]n(n-1)/2$. If $\ell$ is a regular
Hodge type for $\GL_n$, and $t : I_K \to \GL_n(\ov\Q_p)$ is an inertial type,
the ring $\tilde{R}^{\square,\psi}(\ell,t,\ov{r})$ is a quotient of
$\tilde{R}^{\square,\psi}(\ov{r})$ of dimension $d$.
So $\spec(\tilde{R}^{\square,\psi}(\ell,t,\ov{r})/\varpi)$ is a closed
subscheme of $\tilde{X}(\ov{r})$, and it is equidimensional of
dimension $d-1$ (or is empty) by \cite{Kis08}.

Note that
$\spec(\tilde{R}^{\square,\psi}(\ell,t,\ov{r})/\varpi)$, seen as a subscheme of
$\tilde{X}(\ov{r})$,
does not depend on $\psi$ as long as it is compatible
to $\ell,t$. 
Indeed, if $\psi$ and $\psi'$ both satisfy the compatibility condition,
then $\psi/\psi'$ is unramified, 
and the canonical isomorphism between
$\tilde{R}^{\square,\psi}(\ov{r})$
and
$\tilde{R}^{\square,\psi'}(\ov{r})$
induces an isomorphism between
$\tilde{R}^{\square,\psi}(\ell,t,\ov{r})$
and
$\tilde{R}^{\square,\psi'}(\ell,t,\ov{r})$.

Hence we can consider the $d$-dimensional cycle
$Z(\tilde{R}^{\square,\psi}(\ell,t,\ov{r})/\varpi)$, and also
the $(d-1)$-dimensional cycle
$Z(\tilde{R}^{\square,\psi,cr}(\ell,t,\ov{r})/\varpi)$.
These are cycles on $\tilde{X}(\ov{r})$ that do not depend on $\psi$
as long as it is compatible to $\ell,t$.

\subsubsection{The $\PGL_n$ case}

We can do the same construction for $\PGL_n$. 

Let $\ov{\rho} : \Gamma_K \to \PGL_n(\F)$ be a continuous representation,
let $R^{\square}(\ov{\rho})$ be the ring classifying framed
deformations of $\ov{\rho}$.
Let $\X(\ov{\rho}) = \spec(R^{\square}(\ov{\rho}))$
and let
$X(\ov{\rho}) = \spec(R^{\square}(\ov{\rho})/\varpi)$
be its special fiber.

The isomorphism of Theorem \ref{defcadrrings} gives a canonical
isomorphism between 
$\tilde{\X}(\ov{r})$ and $\X(\ov{\rho})$,
and between
$\tilde{X}(\ov{r})$ and $X(\ov{\rho})$,
when $\ov{r}$ is a lift of $\ov\rho$.

Let $d = n^2 + [K:\Q_p]n(n-1)/2$ as before. If $\lambda$ is a regular
Hodge type for $\PGL_n$, and $\tau : I_K \to \PGL_n(\ov\Q_p)$ is an inertial type,
the ring $R^{\square}(\lambda,\tau,\ov{\rho})$ is a quotient of
$R^{\square}(\ov{\rho})$ of dimension $d$ (this follows either from
Theorem \ref{compdefrings}, or from \cite[Prop 4.1.5]{Bal}).
So $\spec(R^{\square}(\lambda,\tau,\ov{\rho})/\varpi)$ is a closed
subscheme of $X(\ov{\rho})$, which is equidimensional of
dimension $d-1$ (or is empty).

\subsubsection{Cycles for $\GL_n$ and $\PGL_n$}
\label{cyclesglnpgln}

Let $\ov{\rho} : \Gamma_K \to \PGL_n(\F)$ be a continuous representation,
and let $\ov{r_0} : \Gamma_K \to \GL_n(\F)$ 
be a lift of $\ov\rho$.

In what follows we identify 
$\tilde{\X}(\ov{r_0})$ with $\X(\ov{\rho})$,
and also
$\tilde{X}(\ov{r_0})$ with $X(\ov{\rho})$.

\begin{theo}
\label{cyclesdefrings}
Let $\lambda$ be a regular Hodge type for $\PGL_n$, and $\tau$ an
inertial type for $\PGL_n$. We set
$d = n^2 + [K:\Q_p]n(n-1)/2$.
Let $\ell$ be a lift of $\lambda$ to a Hodge type for $\GL_n$, and let
$t$ be a fixed inertial type for $\GL_n$ lifting $\tau$. Let $\psi$ be a
character compatible to $t$ and $\ell$.
Then
$$
Z_d(R^{\square}(\lambda,\tau,\ov\rho)) = 
\sum_{\alpha\in \Qq(\tau)}
Z_d(\tilde{R}^{\square,\psi}(\ell,t\otimes\alpha,\ov{r_0}))
$$
as cycles on $\X(\ov{\rho})$,
and the cycles 
$Z_d(\tilde{R}^{\square,\psi}(\ell,t\otimes\alpha,\ov{r_0}))$
for $\alpha\in \Qq(\tau)$ have disjoint supports.

Moreover
$$
Z_{d-1}(R^{\square}(\lambda,\tau,\ov\rho)/\varpi) = 
\sum _{\alpha\in \Qq(\tau)}
Z_{d-1}(\tilde{R}^{\square,\psi}(\ell,t\otimes\alpha,\ov{r_0})/\varpi)
$$
as cycles on $X(\ov\rho)$.
\end{theo}

\begin{proof}
The points of $\spec(R^{\square}(\lambda,\tau,\ov\rho))$ and 
$\spec(\tilde{R}^{\square,\psi}(\ell,t\otimes\alpha,\ov{r_0}))$
of dimension $d$ are in
characteristic zero, as all the rings involved are flat over $\Z_p$. So
the first assertion is a rephrasing of Theorem \ref{compdefrings0}.

The second assertion then follows from Theorem
\ref{compdefrings} and 
\cite[Prop. 2.2.13]{EG14}, applied
to the $R^{\square}(\lambda,\tau,\ov\rho)$-module $\prod_{\alpha\in \Qq(\tau)}
\tilde{R}^{\square,\psi}(\ell,t\otimes\alpha,\ov{r_0})$. 
\end{proof}

\section{Breuil-Mézard conjectures: from $\GL_n$ to $\PGL_n$}
\label{BM}

The Breuil-Mézard conjectures describe the cycles attached to potentially
semi-stable deformation rings for representations with values in $\GL_n$
in terms of some multiplicities of automorphic nature. We state some
similar conjectures for $\PGL_n$, and show that they are consequences of
the conjectures (or known results) for $\GL_n$.

\subsection{Representations of $\GL_n$ and $\SL_n$}
\label{reprGLSL}

\subsubsection{Algebraic representations of $\GL_n$ and $\SL_n$}
\label{algreprGLSL}

The reference for this Section is \cite{Jan}.

We briefly recall the theory of algebraic representations of $\GL_n$ and
$\SL_n$. Here $G$ denotes either $\GL_n$ or $\SL_n$ or $\PGL_n$, and $B$ is the Borel
subgroup of upper triangular matrices. Let $T$ be the (maximal) torus of
diagonal matrices.
We add an index $\GL_n$ or $\SL_n$ or $\PGL_n$ to the objects to distinguish
between the groups when necessary. We see all these groups as
algebraic groups over $\Z_p$.

Let $X^*(T)$ be the group of characters of the torus $T$, and $X^*(T)_+$ the
set of dominant characters with respect to $B$. Denote by $X_*(T)$ the
group of cocharacters of the torus $T$.

Let $\mu \in X^*(T)$, we see it as a character of $B$ via the
canonical projection $B \to T$. We define the dual Weyl module as the
following finite-dimensional representation of $G$:
$$
W_G(\mu) = \ind_{B}^G(w_0\mu)
$$
where $w_0$ is the longest element of the Weyl group
(observe that $W(\mu) = 0$ if $\mu$ is not in $X^*(T)_+$).

Recall that $X^*(T_{\GL_n})$ is isomorphic to $\Z^n$, and $X^*(T_{\SL_n})$
is isomorphic to $\Z^n/\Z\ul{1}$ (where $\ul{1} = (1,\dots,1)$), with the
natural map $T_{\SL_n} \to T_{\GL_n}$ inducing the natural projection
$\Z^n \to \Z^n/\Z\ul{1}$. If $\mu \in X^*(T_{\GL_n})$ we denote by
$\ov{\mu}$ its image in $X^*(T_{\SL_n})$. Note that $\mu$ is
dominant if and only if $\ov\mu$ is.

The following isomorphism is easily checked for $\mu \in
X^*(T_{\GL_n})_+$:
\begin{equation}
\label{restrsln}
W_{\GL_n}(\mu)_{|\SL_n} = W_{\SL_n}(\ov\mu)
\end{equation}

Let $R$ be the set of roots for $\GL_n$ relative to $T$. 
It is the set of characters
$e_i-e_j$ for $1 \leq i,j \leq n$, where $e_i$ is the character sending
a diagonal matrix to its $i$-th diagonal coefficient. 
With our choice of Borel subgroup, the
positive roots are the $e_i-e_j$ for $i < j$, and the positive simple
roots are the $e_i-e_{i+1}$. If $\alpha$ is a root, its attached coroot
$\alpha^\vee$ is equal to $\alpha$. For $\SL_n$, the roots, positive
roots, positive simple roots are simply the image in $X^*(T_{\SL_n})$ of
the objects for $\GL_n$.

\subsubsection{Serre weights}

Let $\Sigma$ be the set of Serre weights for $\SL_n(\F_q)$, that is, the
set of isomorphism classes of irreducible representations of
$\SL_n(\F_q)$ with coefficients in $\ov\F_p$.

Let $\s$ be the set of Serre weights of $\GL_n(\F_q)$.

The kernel of the reduction map $\GL_n(\O_K) \to \GL_n(\F_q)$ is a
pro-$p$-subgroup, and similarly for the kernel of 
$\SL_n(\O_K) \to \SL_n(\F_q)$. This means that any irreducible
representation of $\GL_n(\O_K)$ with coefficients in $\ov\F_p$ factors
through $\GL_n(\F_q)$, so we can also see $\s$ as the set of isomorphism
classes of irreducible representations of $\GL_n(\O_K)$ on $\ov\F_p$, and $\Sigma$ as 
as the set of isomorphism classes of irreducible representations of
$\SL_n(\O_K)$ on $\ov\F_p$.

We have the following result:

\begin{theo}
\label{serreweights}
If $s$ is an irreducible representation of $\GL_n(\F_q)$, then its
restriction to $\SL_n(\F_q)$ is still irreducible. The map $r : \s \to
\Sigma$ given by restriction is surjective. If $\sigma$ is the
restriction of $s$, then $r^{-1}(\sigma) = \{s \otimes \det^m, 0 \leq m
< q-1\}$ and these $q-1$ representations are distinct.
\end{theo}

In order to prove this, we recall the general construction of irreducible
representations of $\GL_n(\F_q)$ and $\SL_n(\F_q)$. 

Let $G$ be $\GL_n$ or $\SL_n$ and let $T$ be as before.
Let $\mu \in X^*(T)_+$. 
Denote by $W_G(\mu)_{\ov\F_p}$ the representation of $G(\ov\F_p)$
obtained by evaluating
$W_G(\mu)$ at $\ov\F_p$.

Let $F_G(\mu)$ be the socle of $W_G(\mu)_{\ov\F_p}$. Then the set of irreducible
$G(\ov\F_p)$-representations is exacly the set of the $F_G(\mu)$, $\mu \in
X^*(T)_+$. From equation (\ref{restrsln}) we get:

\begin{lemm}
\label{restrbarFp}
Let $\mu \in X^*(T_{\GL_n})_+$ and let $\ov\mu$ be its image in  $X^*(T_{\SL_n})_+$. Then
$W_{\SL_n}(\ov\mu) = W_{\GL_n}(\mu)_{|\SL_n(\ov\F_p)}$, and
$F_{\SL_n}(\ov\mu) = F_{\GL_n}(\mu)_{|\SL_n(\ov\F_p)}$.
\end{lemm}

We denote by $X_s(T)$ the set of $p^s$-restricted weights. It is the set
$$
\{\mu \in X^*(T), 0 \leq \langle \mu,\alpha^\vee \rangle < p^s\text{ for all
simple roots }\alpha\}
$$

So for $G = \GL_n$, $X_s(T_{\GL_n})$ is the set of $n$-uples $(\mu_1,\dots,\mu_n) \in
\Z^n$ with $\mu_1 \geq \dots \geq \mu_n$, and $\mu_i - \mu_{i+1} < p^s$
for all $i$. For $G = \SL_n$, $X_s(T_{\SL_n})$ is the image in
$X^*(T_{\SL_n})$ of $X_s(T_{\GL_n})$.

Then we have
the following result (the case of $\SL_n$ follows from \cite{Jan}, which
proves the result for any semisimple group, and the case of $\GL_n$
follows from \cite[Theorem 3.10]{Her}):

\begin{theo}
\label{descrirr}
Assume that $G = \GL_n$ or $G = \SL_n$. Then we obtain all the
irreducible representations of $G(\F_{p^s})$ as follows: let $\mu \in
X_s(T)$, and take the restriction of $F_G(\mu)$ to $G(\F_{p^s})$, then
this is irreducible, and any irreducible representation of $G(\F_{p^s})$ is
isomorphic to exactly one of these representations.
\end{theo}

From this and from Lemma \ref{restrbarFp}, we get Theorem
\ref{serreweights}.

\subsection{Types and the local Langlands correspondence}
\label{deftypesn}

Our goal is to define types for $\SL_n$ in the context of the
Breuil-Mézard conjecture. A clean definition would rest on an inertial
Langlands correspondence for $\SL_n$ which is not known at present. So we
give a construction relying on the results for $\GL_n$.

\subsubsection{Types for $\GL_n$}

Denote by $\text{rec}_p$ the local Langlands correspondence for $\GL_n$:
if $\pi$ is a smooth irreducible admissible representation of $\GL_n(K)$,
then $\text{rec}_p(\pi)$ is a Weil-Deligne representation of $W_K$.

We recall the following result regarding types for $\GL_n$ (see
\cite[Theorem 3.7]{CEG}:

\begin{defi}
\label{typeLanglands}
Let $t$ be an inertial type for $\GL_n$. We say that a smooth,
irreducible, finite dimensional representation $\sigma(t)$ of
$\GL_n(\O_K)$ over $\ov\Q_p$ is a type for $t$ if it satisfies the
following property:
let $\tilde{t}$
be a Frobenius semi-simple Weil-Deligne representation of $W_K$.
Then
$\tilde{t}_{|I_K}$ is isomorphic to $t$ and has $N=0$
if and only if the restriction to
$\GL_n(\O_K)$ of $\text{rec}_p^{-1}(\tilde{t})\otimes |\det|^{(n-1)/2}$
contains a copy of $\sigma(t)$.
\end{defi}

\begin{theo}
\label{typeexists}
For any inertial type $t$ for $\GL_n$, there exists a type $\sigma(t)$
for $t$.
\end{theo}

Uniqueness is not known in general (but it is known for $n=2$, see
\cite[Appendix]{BM}).

We denote by $\sgl(t)$ a choice of type for each inertial type $t$ for
$\GL_n$.
By the compatibility of the local Langlands
correspondence with twist by characters, we can assume moreover the
following property:
Let $\alpha$ be a smooth character $I_K \to \ov\Q_p^\times$. 
Then $\sgl(t \otimes \alpha) = \sgl(t) \otimes \left(\alpha \circ
\det\right)$, where
we see $\alpha$ as a character of $\O_K^\times$ via local class field
theory.

\subsubsection{Types modulo $p$ for $\SL_n$}

Let $\tau: I_K \to \PGL_n(\ov\Q_p)$ be an inertial type. Our goal is to
attach to $\tau$ a smooth, semi-simple representation $\ov{\ssl(\tau)}$ of $\SL_n(\O_K)$
with coefficients in $\ov\F_p$.

From the choice of $\sgl$ we deduce:

\begin{lemm}
Let $\tau : I_K \to \PGL_n(\ov\Q_p)$ be an inertial type. Then 
$\sgl(t)_{|\SL_n(\O_K)}$ does not depend on the choice of the inertial
type $t$ for $\GL_n$ lifting $\tau$.
\end{lemm}

Let $g$ be a element of $\GL_n(\O_K) \setminus \O_K^\times\SL_n(\O_K)$.
If $\sigma$ is a representation of $\SL_n(\O_K)$, we denote by $\sigma^g$
the representation conjugate by $g$, that is, the representation
$x \mapsto \sigma(gxg^{-1})$.

We have the following result:

\begin{prop}
\label{dectype}
Let $t$ be an inertial type for $\GL_n$, and $\tau$ the corresponding
inertial type for $\PGL_n$.

Then
$\sgl(t)_{|\SL_n(\O_K)}$ is a direct sum of copies of $M$ distinct irreducible representations
$\pi_1,\dots,\pi_M$ of $\SL_n(\O_K)$, with $\pi_i$ appearing with
the multiplicity $m_i$, satisfying the following
properties:

\begin{enumerate}
\item
all the $m_i$ are equal to some value $m$.

\item $Mm^2 = \# \Gg_I(\tau)$, where $\Gg_I(\tau)$ is the group defined
in Paragraph \ref{twists}.

\item the total number of representations that appear, counting
multiplicities, is $Mm = \# \Gg_I(\tau)/m$.

\item for all $i$, there exists $g_i \in \GL_n(\O_K)$ such that $\pi_i
= \pi_1^{g_i}$.

\item the representations $\ov{\pi_i}^{ss}$ are all isomorphic.

\item
write the decomposition of $\ov{\pi_i}^{ss}$ as a sum of irreducible
representations as $\sum_{\sigma\in \Sigma} n_\sigma [\sigma]$. 
Then $m \mid n_\sigma$ for all $\sigma \in \Sigma$.
\end{enumerate}
\end{prop}

By $\ov{\pi}^{ss}$ we denote the semi-simplification of the reduction
modulo $p$ of $\pi$.

\begin{proof}
Properties (1) to (4) are a consequence of \cite[Lemma 2.1]{GK} applied
to $G = \GL_n(\O_K)$, and $H = \O_K^\times\SL_n(\O_K)$ and the
irreducible smooth representation $\sigma = \sgl(t)$ of $\GL_n(\O_K)$
(although our case is simpler, as we are actually dealing with
representations of finite groups).

The set denoted by $X(\sigma)$ in \cite[Lemma 2.1]{GK} of smooth
characters $\alpha$ of $\GL_n(\O_K)$ such that $\sigma \simeq \sigma
\otimes (\alpha \circ \det)$ corresponds via local class field theory to the
set of smooth characters of inertia $\alpha$ such that $t \simeq t
\otimes \alpha$, that is, to $\Gg_I(\tau)$. 
 
Let us prove (5).
Note that $\ov{\pi_i}^{ss} = \ov{\pi_1^{g_i}}^{ss}$. But any irreducible
representation of $\SL_n(\F_q)$ extends to a representation of
$\GL_n(\F_q)$ by Theorem \ref{serreweights}, so two semi-simple
representations of $\SL_n(\F_q)$ that are conjugate by $g_i \in
\GL_n(\F_q)$ are isomorphic.

We defer the proof of (6) to Proposition \ref{integred}.
\end{proof}

%

\begin{rema}
\label{multfree}
Note that when $\gcd(n,q-1)$ is square-free, we have $m=1$ and the total
number of representations that appear, counting multiplicities, is $\#
\Gg_I(\tau)$ in the proposition above. Indeed, $\# \Gg_I(\tau)$ divides 
$\# \Gg_I$ which is equal to $\gcd(n,q-1)$ as explained in Paragraph
\ref{twists}.
\end{rema}

\begin{defi}
\label{typesln}
Let $\tau: I_K \to \PGL_n(\ov\Q_p)$ be an inertial type. We define the
representation $\ov{\ssl(\tau)}$ of $\SL_n(\O_K)$ as follows: let $t :
I_K \to \GL_n(\ov\Q_p)$ be an inertial type lifting $\tau$.
Let $\pi$ be one of the irreducible components of
$\sgl(t)_{|\SL_n(\O_K)}$.

We set $\ov{\ssl(\tau)} = \frac{1}{m}\ov{\pi}^{ss}$, where $m$ is as in Proposition
\ref{dectype}.
\end{defi}

\begin{rema}
\label{typemodp}
It follows from Proposition \ref{dectype} that the isomorphism class of
$\ov{\ssl(\tau)}$ is canonically defined up to the choice of $\sgl$, 
even though we had to make a
choice in the definition, and moreover it is really a representation, and
not only a rational element in the Grothendieck group of representations
of $\SL_n(\O_K)$.

For these reasons we choose to define only a representation $\ov{\ssl(\tau)}$
in characteristic $p$, and we will not define a representation 
$\ssl(\tau)$ in characteristic $0$. 
\end{rema}

\subsection{Multiplicities}

\subsubsection{Representations coming from the Hodge type}
\label{repHT}

We construct a representation of $\GL_n(\O_K)$ (resp. $\SL_n(\O_K)$)
with coefficients in $\O_E$ for some finite extension $E$ of $\Q_p$
containing all conjugates of $K$,
from a regular Hodge type $\ell$ for $\GL_n$ (resp. for
$\PGL_n)$).

We use the notation of Paragraph \ref{algreprGLSL}.

Denote by $\delta$ the element $(n-1,n-2,\dots,0) \in X^*(T_{\GL_n})$,
and $\ov\delta$ its image in $X^*(T_{\SL_n})$. We make the following
observation: if $\mu \in X^*(T_{\GL_n})$ is dominant and regular,
then $\mu-\delta$ is dominant, and similarly for $\SL_n$.

The construction for $\GL_n$ is as follows: the regular
Hodge type $\ell$ is an element of $X^*(T_{\GL_n})_+^\e$ such that for
all $\iota\in\e$, $\ell_\iota$ is regular. For each
$\iota \in \e$, consider the algebraic representation
$W_{\GL_n}(\ell_\iota-\delta)$. Then we can evaluate it at $\O_K$ to get a
finite-rank representation $M_{\ell,\iota}$ of $\GL_n(\O_K)$. Finally we
set $\sgl(\ell) = \otimes_{\iota\in \e}
\left(M_{\ell,\iota}\otimes_{\O_K,\iota}\O_E\right)$,
seen as a representation of $\GL_n(\O_K)$.

The construction for $\SL_n$ is the same, the Hodge type $\lambda$
for $\PGL_n$ being an element of $X^*(T_{\SL_n})_+^\e$. We denote by
$\ssl(\lambda)$ the representation of $\SL_n(\O_K)$ that we obtain.

Observe that by construction, $\ssl(\lambda)$ is the restriction to
$\SL_n(\O_K)$ of $\sgl(\ell)$ if $\ell$ lifts $\lambda$.

The construction for $\GL_n$ coincides with the one in \cite[\S 4]{EG14}
(note that we have different conventions for the Hodge-Tate weights).

\subsubsection{Multiplicities coming from the Hodge type and inertial type}
\label{BMmult}

Let $\lambda$ be a Hodge type
and $\tau$ be an inertial type for $\PGL_n$.
We define a family $(\alpha_{\lambda,\tau}(\sigma))_{\sigma\in\Sigma}$ of
non-negative rational numbers as follows:  Consider the
representation $\ov{\ssl(\tau)} \otimes \ov{\ssl(\lambda)}$ of $\SL_n(\O_K)$. Then
its semi-simplification factors as a
finite-dimensional representation of $\SL_n(\F_q)$, which can we written
as
$$
\left(\ov{\ssl(\tau)}\otimes\ov{\ssl(\lambda)}\right)^{ss} = \sum_{\sigma\in
\Sigma}\alpha_{\lambda,\tau}(\sigma)[\sigma].
$$
We will see in Proposition \ref{integred} that the
$(\alpha_{\lambda,\tau}(\sigma))_{\sigma\in\Sigma}$ are actually
integers.

Let $t$ be an inertial type for $\GL_n$, and $\ell$ a Hodge type for
$\GL_n$.  Define the family
$(a_{\ell,t}(s))_{s\in\s}$ of
non-negative integers by the formula:
$$
\ov{\sgl(t)\otimes\sgl(\ell)}^{ss} = \sum_{s\in
\s}a_{\ell,t}(s)[s]
$$

The key result is the following:

\begin{prop}
\label{integred}
Let $t$ be an inertial type for $\GL_n$ lifting $\tau$, and $\ell$ a Hodge type for
$\GL_n$ lifting $\lambda$.  

Let $\s(\lambda,\tau)$ be the subset of $\s$ of representations that have the
same central character as $\ov{\sgl(t)\otimes\sgl(\ell)}^{ss}$ (this does
not depend on the choice of $t$ or $\ell$). Then
\begin{enumerate}

\item
$a_{\ell,t\otimes\mu}(s) = 0$ for any $s$ not in $\s(\lambda,\tau)$ and
for any $\mu \in \Gg_I$.

\item
If $s$ and $s'$ have the same restriction to
$\SL_n(\F_q)$ and the same central character then
$
\sum_{\mu \in \Gg_I/\Gg_I(\tau)}
a_{\ell,t\otimes\mu}(s)
=
\sum_{\mu \in \Gg_I/\Gg_I(\tau)}
a_{\ell,t\otimes\mu}(s')
$.

\item
$
\alpha_{\lambda,\tau}(\sigma)
= 
\sum_{\mu \in \Gg_I/\Gg_I(\tau)}
a_{\ell,t\otimes\mu}(s)$ for any 
$s \in \s(\lambda,\tau)$ restricting to $\sigma$.

\end{enumerate}

In particular, 
the $\alpha_{\lambda,\tau}(\sigma)$ are in $\Z_{\geq 0}$, and 
$\ov{\ssl(\tau)}$ is integral in the Grothendieck group of
representations of 
$\SL_n(\O_K)$.
\end{prop}

Recall that $\Gg_I$ and $\Gg_I(\tau)$ are defined in Paragraph
\ref{twists}. The last statement of this Proposition is statement (6) of
Proposition \ref{dectype}.

\begin{proof}
(1) is clear.

Let us prove (2). Observe first that the isomorphism class of
$t\otimes\mu$ for $\mu\in\Gg_I$ depends only on the class of $\mu$ in 
$\Gg_I/\Gg_I(\tau)$, so $a_{\ell,t\otimes\mu}(s)$ is well-defined for
$\mu\in\Gg_I/\Gg_I(\tau)$.

Let $g = \gcd(q-1,n)$ be the cardinality of $\Gg_I$, and let $h$ be the
cardinality of $\Gg_I(\tau)$.
Let $\alpha_i$ the character $\O_K^\times \to \ov\Q_p^\times$ defined by the
composition of the map $\O_K^\times \to \mu_{q-1}(K)$ and the map $x
\mapsto x^{i(q-1)/g}$. We also denote by $\alpha_i$ the corresponding character of
the inertia subgroup $I_K$ given by local class field theory.  Then
$\alpha_1$ generates $\Gg_I$ and $\alpha_{g/h}$ generates $\Gg_I(\tau)$.

So $a_{\ell,t\otimes\alpha_i}(s)$ depends only on $i$ modulo $g/h$ as $t
\otimes \alpha_{g/h}$ is isomorphic to $t$,
and we can write
$
\sum_{\mu \in \Gg_I/\Gg_I(\tau)}
a_{\ell,t\otimes\mu}(s)
$
as
$
\sum_{i=0}^{g/h-1}
a_{\ell,t\otimes\alpha_i}(s)
$.

Now, we can write $s'$ as $s \otimes \left(\alpha_j\circ \det\right)$ for some $j$.
Then
$$
\sum_{\mu \in \Gg_I/\Gg_I(\tau)}
a_{\ell,t\otimes\mu}(s')
=
\sum_{i=0}^{g/h-1}
a_{\ell,t\otimes\alpha_i}(s\otimes \alpha_j\circ \det)
=
\sum_{i=0}^{g/h-1}
a_{\ell,t\otimes\alpha_{i-j}}(s)
=
\sum_{i=0}^{g/h-1}
a_{\ell,t\otimes\alpha_i}(s)
$$
which is what we wanted.

Let us prove (3).

As representation of $\GL_n(\F_q)$, we have
$$
\oplus_{\mu \in \Gg_I/\Gg_I(\tau)}
\ov{\sgl(t\otimes\mu)\otimes\sgl(\ell)}^{ss}
=
\sum_{s\in\s(\lambda,\tau)}
\sum_{\mu \in \Gg_I/\Gg_I(\tau)}
a_{\ell,t\otimes\mu}(s)[s]
$$
We restrict this to $\SL_n(\F_q)$, then the coefficient appearing before
$[\sigma]$ for $\sigma\in \Sigma$
is
$
\sum_{s\in\s(\lambda,\tau),s_{|\SL_n(\F_q)=\sigma}}
\sum_{\mu \in \Gg_I/\Gg_I(\tau)}
a_{\ell,t\otimes\mu}(s)
$
which is
$g\sum_{\mu \in \Gg_I/\Gg_I(\tau)}
a_{\ell,t\otimes\mu}(s)$ for any choice of 
$s \in \s(\lambda,\tau)$ restricting to $\sigma$ (as there are $g$ of
them).

On the other hand we compute directly
$
\oplus_{\mu \in \Gg_I/\Gg_I(\tau)}
\ov{\sgl(t\otimes\mu)\otimes\sgl(\ell)}^{ss}
$
as a representation of $\SL_n(\F_q)$.
In the notation of Proposition \ref{dectype},
recall that
$\sgl(t\otimes\mu)_{|\SL_n(\O_K)}$ does not depend on $\mu$ and is the sum of $h/m$ irreducible
representations $\pi_i$ of $\SL_n(\O_K)$ (not necessarily distinct), such that $\ov{\pi_i}^{ss}$
does not depend on $i$, and $\ov{\ssl(\tau)}$ is $(1/m)\ov{\pi_i}^{ss}$ for some $i$.
This means that
\begin{eqnarray*}
\oplus_{\mu \in \Gg_I/\Gg_I(\tau)}
\ov{\sgl(t\otimes\mu)\otimes\sgl(\ell)}^{ss}
&=&
(g/h)\ov{\oplus_{i=1}^{h/m}\pi_i\otimes\ssl(\lambda)}^{ss} \\
&=&
(g/h)(h/m)\ov{\pi_1\otimes\ssl(\lambda)}^{ss}\\
&=&
g\ov{\ssl(\tau)\otimes\ssl(\lambda)}^{ss}\\
&=& g\sum_{\sigma\in \Sigma}\alpha_{\lambda,\tau}(\sigma)[\sigma]
\end{eqnarray*}
Comparing the two results, we get what we wanted.

\medskip

Finally take $\lambda = 0$. We get that
$
\ov{\ssl(\tau)} = 
\sum_{\sigma\in \Sigma}\alpha_{0,\tau}(\sigma)[\sigma]
= 
\sum_{\sigma\in \Sigma}
\left(\sum_{\mu \in \Gg_I/\Gg_I(\tau)}
a_{0,t\otimes\mu}(s)\right)[\sigma]$ 
where for each $\sigma$ we choose any $s \in \s(\lambda,\tau)$
restricting to $\sigma$. So 
$ \ov{\ssl(\tau)}$ is an integral combination of the representations
$[\sigma]$ for $\sigma \in \Sigma$.
\end{proof}

\subsection{Geometric Breuil-Mézard in the potentially crystalline case}
\label{geomBM}

Fix a continuous representation $\ov\rho: \Gamma_K \to \PGL_n(\ov\F_p)$.
As before we denote by $R^\square(\ov\rho)$ the ring classifying framed
deformations of $\ov\rho$. 

Let $\ov{r} : \Gamma_K \to \GL_n(\ov\F_p)$ 
be a lift of $\ov\rho$.
In what follows we identify 
$\tilde{\X}(\ov{r})$ with $\X(\ov{\rho})$,
and also
$\tilde{X}(\ov{r})$ with $X(\ov{\rho})$
as in Paragraph \ref{cyclesglnpgln}.

\subsubsection{Conjecture}

We now state a Breuil-Mézard conjecture for representations with values
in $\PGL_n$.

\begin{conj}
\label{geomBMPGLnconj}
Let $n \geq 2$, and let $p$ be a prime number that does not divide $n$.
Let $\ov\rho : \Gamma_K \to \PGL_n(\ov\F_p)$ be a continuous
representation. 

Then there exist cycles $\cc_{\sigma}(\ov\rho)$ on $X(\ov\rho)$ such
that,
for any inertial type $\tau : I_K \to \PGL_n(\ov\Q_p)$ and for any
regular Hodge type $\lambda$ for $\PGL_n$,
we have:
$$
Z(R^{\square,cr}(\lambda,\tau,\ov\rho)/\varpi) =
\sum_{\sigma\in\Sigma}\alpha_{\lambda,\tau}(\sigma)\cc_{\sigma}(\ov\rho)
$$
as cycles on $X(\ov\rho)$, where the integers $\alpha_{\lambda,\tau}(\sigma)$ 
are defined as in Section \ref{BMmult}.
\end{conj}

\subsubsection{Statement of the result}

We now show that Conjecture \ref{geomBMPGLnconj} holds as soon as the
corresponding conjecture for $\GL_n$ holds.

\begin{theo}
\label{geomBMPGLn}
Let $n \geq 2$, and let $p$ be a prime number that does not divide $n$.

Let $\ov\rho : \Gamma_K \to \PGL_n(\ov\F_p)$ be a continuous
representation. Let $\tau : I_K \to \PGL_n(\ov\Q_p)$ be an inertial type,
and let $\lambda$ be a regular Hodge type for $\PGL_n$.

Assume that there exists a lift $\ov{r} : \Gamma_K \to \GL_n(\ov\F_p)$ of
$\ov\rho$, satisfying the following condition:
for some lift $\ell$ of $\lambda$, 
for some inertial type $t : I_K \to \GL_n(\ov\Q_p)$ lifting $\tau$, 
for some character $\psi$ compatible to $t$ and $\ell$, 
and for any $\mu \in \Gg_I$
we have the equality:
$$
Z(R^{\square,cr,\psi}(\ell,t\otimes\mu,\ov{r})/\varpi) =
\sum_{s\in\s}a_{\ell,t\otimes\mu}(s)\cc_s(\ov{r})
$$

Then, setting 
$$
\cc_\sigma(\ov\rho) = \sum_{s\in\s, s_{|\SL_n(\F_q)} = \sigma}
\cc_s(\ov{r})
$$ 
we have:
$$
Z(R^{\square,cr}(\lambda,\tau,\ov\rho)/\varpi) =
\sum_{\sigma\in\Sigma}\alpha_{\lambda,\tau}(\sigma)\cc_{\sigma}(\ov\rho)
$$
as cycles on $X(\ov\rho)$,
where the integers $\alpha_{\lambda,\tau}(\sigma)$ 
and $a_{\ell,t}(s)$
are defined as in
Section \ref{BMmult}.
\end{theo}

\subsubsection{Known cases}
\label{knowncr}

The equality of cycles for $\GL_n$ constitutes the Breuil-Mézard
conjecture for $\GL_n$, stated in \cite{BM} for $n=2$, $K = \Q_p$, and in
general in \cite{EG14}.
Let us describe a few situations where the equality of cycles for $\GL_n$
is known, and hence we get also the equality of cycles for $\PGL_n$.

First, in the case $n=2$, $p \neq 2$, $K = \Q_p$ the equality of cycles is know for
all $\ov{r}$, $\ell$ and $t$, by
\cite{Kis09,Pas15,HT15,Pas16,Tun18,Tun19}. 
Moreover, the cycles
$\cc_s(\ov{r})$ and $\cc_\sigma(\ov\rho)$ can be described explicitly: 
For $\GL_2$,
a Serre weight in this case is a representation
$s = \text{Symm}^a\ov{\F}_p^2\otimes \det^b$ with $0 \leq b < p-1$, and
$0 \leq a \leq p-1$. For such an $s$ set $\ell_s = (a+b+1,b)$, then
$\cc_s(\ov{r}) = 
Z(R^{\square,cr,\psi}(\ell_s,\mathbf{1},\ov{r})/\varpi)$
where $\mathbf{1}$ is the trivial inertial type.
For $\PGL_2$,
a Serre weight in this case is a representation 
$\sigma = \text{Symm}^a\ov{\F}_p^2$ with 
$0 \leq a \leq p-1$. For such a $\sigma$ set $\lambda_\sigma$ to be the
class of $(a+1,0)$, then
$\cc_\sigma(\ov{\rho}) = 
Z(R^{\square,cr}(\lambda_\sigma,\mathbf{1},\ov{\rho})/\varpi)$. In
particular this cycle is independent of the choice of $\ov{r}$ lifting
$\ov\rho$ and of the choice of $\psi$.

In the case $n=2$, $p \neq 2$, but without restriction on $K$, the
equality of cycles in known for all $\ov{r}$ and inertial type $t$, and
for $\ell$ corresponding to potentially Barsotti-Tate representations,
that is $\ell_\iota = (1,0)$ for all $\iota\in\e$ by the results of
\cite{GK14}, and also for trivial inertial type and some values of the
Hodge-Tate weights by \cite{Bart}.


There are also partial results for larger $n$ and unramified $K$, with
$t$ tamely ramified and some genericity condition on $\ov{r}$ in \cite{LLLM}.


\subsubsection{Proof of Theorem \ref{geomBMPGLn}}


\begin{proof}
Fix some inertial type $t$ lifting $\tau$, $\ell$ a Hodge type lifting
$\lambda$, and $\psi$ a character
compatible to $t$ and $\ell$.
By Theorem \ref{cyclesdefrings}, we have
$$
Z(R^{\square,cr}(\lambda,\tau,\ov\rho)/\varpi) =
\sum_{\mu \in \Gg_I/\Gg_I(\tau)}
Z(R^{\square,cr,\psi}(\ell,t\otimes\mu,\ov{r})/\varpi)
=
\sum_{\mu \in \Gg_I/\Gg_I(\tau)}
\sum_{s\in\s}a_{\ell,t\otimes\mu}(s)\cc_s(\ov{r})
$$

By definition of the cycles $\cc_\sigma(\ov\rho)$ and 
using Proposition \ref{integred} we get:
\begin{eqnarray*}
\sum_{\mu \in \Gg_I/\Gg_I(\tau)}
\sum_{s\in\s}a_{\ell,t\otimes\mu}(s)
\cc_s(\ov{r}) 
&=&
\sum_{\mu \in \Gg_I/\Gg_I(\tau)}
\sum_{s\in\s(\lambda,\tau)}a_{\ell,t\otimes\mu}(s)
\cc_s(\ov{r}) \\
&=& 
\sum_{\sigma \in \Sigma}
\left(\sum_{\mu \in \Gg_I/\Gg_I(\tau)}
\sum_{s\in\s(\lambda,\tau),s_{|\SL_n(\F_q)}=\sigma}
a_{\ell,t\otimes\mu}(s)\right)
\cc_s(\ov{r}) \\
&=&
\sum_{\sigma \in \Sigma}
\left(
\sum_{\mu \in \Gg_I/\Gg_I(\tau)}
a_{\ell,t\otimes\mu}(s)\right)
\cc_\sigma(\ov{\rho})
\end{eqnarray*}
where for each $\sigma\in \Sigma$ we choose any $s \in \s(\lambda,\tau)$
lifting $\sigma$.

So we get that
$$
\sum_{\mu \in \Gg_I/\Gg_I(\tau)}
\sum_{s\in\s}a_{\ell,t\otimes\mu}(s)
\cc_s(\ov{r}) 
= 
\sum_{\sigma\in\Sigma}
\alpha_{\lambda,\tau}(\sigma)
\cc_\sigma(\ov\rho)
$$
which is the result.
\end{proof}

\subsection{The potentially semi-stable case}

We can give a similar result for potentially semi-stable deformation
rings, instead of potentially crystalline deformation rings. In this
situation we need another definition for the types.  We give a precise
statement only for $n=2$, where all the results needed about types are
known, and $K = \Q_p$ so that we can give an unconditional statement.

\subsubsection{Types}

We have the following result in dimension $2$ (see \cite[Appendix]{BM}):

\begin{theo}
\label{typedim2}
Let $t$ be an inertial type for $\GL_2$. There exists a unique smooth
irreducible representation $\sgd(t)$ of $\GL_2(\O_K)$ with coefficients
in $\ov\Q_p$ satisfying the following property:
let $\pi$ be any smooth, irreducible, infinite-dimensional representation
of $\GL_2(K)$, then $\sgd(t)$ is a constituent of the restriction of
$\pi$ to $\GL_2(\O_K)$ if and only if the restriction to $I_K$ of
$\text{rec}_p^{-1}(\pi)$ is isomorphic to $t$.
\end{theo}

\begin{rema}
If $t$ is not a scalar representation, then the type defined above is the
same as the one in Definition \ref{typeLanglands}. When $t$ is a scalar,
that is given by $\alpha \oplus \alpha$ for some smooth character
$\alpha$, the type defined above is $\text{St}\otimes(\alpha \circ
\det)$, where $\text{St}$ is the Steinberg representation, whereas the
type given in Definition \ref{typeLanglands} is $\alpha \circ \det$.
\end{rema}

Let $\tau$ be an inertial type for $\PGL_2$, we define the type
$\ov{\ssd(\tau)}$ from $\sgd(t)$ as in Definition \ref{typesln}. Note that
here there is no denominator in the definition of the type, as noted in
Remark \ref{multfree}.

\bigskip

From now on we take $K = \Q_p$.

\subsubsection{Cycles on deformations rings}

We define cycles $\cc_s(\ov{r})$ on $\tilde{X}(\ov{r})$ for a representation $\ov{r} :
\Gamma_{\Q_p} \to \GL_2(\ov{\F}_p)$
as in Paragraph \ref{knowncr}:
when $s = \text{Symm}^a\ov{\F}_p^2\otimes \det^b$ with $0 \leq b < p-1$, and
$0 \leq a \leq p-1$, set $\ell_s = (a+b+1,b)$, then
$\cc_s(\ov{r}) = 
Z(R^{\square,cr,\psi}(\ell_s,\mathbf{1},\ov{r})/\varpi)$
where $\mathbf{1}$ is the trivial inertial type, and $\psi$ is compatible
to $\ell_s$ and $\mathbf{1}$.

Then we define a family of cycles $\cc_\sigma(\ov\rho)$ for $\sigma \in
\Sigma$. As before we identify $\tilde{X}(\ov{r})$ and $X(\ov\rho)$.
Note that these are the same cycles as those appearing in Conjecture
\ref{geomBMPGLnconj} and Theorem \ref{geomBMPGLn}.

\begin{defi}
\label{defcyclepgl}
We set $\cc_\sigma(\ov\rho) = \sum_{s\in\s, s_{|\SL_2(\F_p)} = \sigma}
\cc_s(\ov{r})$.
\end{defi}

Then as observed in Paragraph \ref{knowncr}:

\begin{prop}
\label{samecycle}
The cycle $\cc_\sigma(\ov\rho)$ does not depend on the choice of $\ov{r}$ and
$\psi$.
\end{prop}

\subsubsection{Statement}

We have the following result for $K = \Qp$.

\begin{theo}
\label{geomBMPGL2Qp}
Let $p > 2$ be a prime number.

Let $\ov\rho : \Gamma_{\Q_p} \to \PGL_2(\ov\F_p)$ be a continuous
representation. Let $\tau : I_{\Q_p} \to \PGL_2(\ov\Q_p)$ be an inertial type,
and let $\lambda \in \Z_{\geq 1}$ be a regular Hodge type for $\PGL_2$.

Then:
$$
Z(R^{\square}(\lambda,\tau,\ov\rho)/\varpi) =
\sum_{\sigma\in\Sigma}\alpha^{st}_{\lambda,\tau}(\sigma)\cc_{\sigma}(\ov\rho)
$$

where the integers $\alpha^{st}_{\lambda,\tau}(\sigma)$ 
are defined as in Section \ref{BMmult}, but using the type attached to
$\tau$ as in Paragraph \ref{typedim2} instead of the one defined in
Paragraph \ref{deftypesn}.
\end{theo}

\begin{proof}
We observe first that we can state an analogue of Theorem
\ref{geomBMPGLn} with potentially semi-stable deformation rings instead
of potentially crystalline ones, with exactly the same proof. Then we use
the fact that the result for representations of $\Gamma_{\Q_p}$ with
values in $\GL_2(\ov\F_p)$ is known, by the references cited in Paragraph
\ref{knowncr}.
\end{proof}


\newcommand{\etalchar}[1]{$^{#1}$}

\end{document}